\documentclass[11pt]{article}

\usepackage{paralist}
\usepackage[margin=3.0cm]{geometry} % Changes the margins of the paper.
\usepackage{graphicx} % Simple graphics.

\usepackage{amssymb, amsfonts, amsmath, amscd, amsthm} % Symbols, fonts, diagrams, and environments from AMS.
\usepackage{mathtools} % Symbols such as inclusion arrows etc.
\usepackage{titlesec} % Section and subsection fonts.
\usepackage{csquotes} % Quotation marks.
\usepackage[shortlabels]{enumitem} % Changes line spacing within lists.
\usepackage{hyperref} % Hyperlinks.
\usepackage{mathrsfs} % For \mathscr
\usepackage{bbm}
\usepackage{multicol}
\usepackage[misc]{ifsym}
\usepackage{nomencl}

\makenomenclature

\renewcommand*{\nompreamble}{\begin{multicols}{2}}
\renewcommand*{\nompostamble}{\end{multicols}}
\hypersetup{
     colorlinks   = true,
     citecolor    = green
}

\numberwithin{equation}{section}

\titleformat*{\section}{\Large \scshape\center}
\titleformat*{\subsection}{\fontsize{14}{14} \sffamily}

\theoremstyle{plain}
\newtheorem{theorem}{Theorem}[section]
\newtheorem*{theorem*}{Theorem}

\newtheorem{lemma}[theorem]{Lemma}
\newtheorem{proposition}[theorem]{Proposition}
\newtheorem{thm*}{Theorem}
\newtheorem{prop*}{Proposition}
\newtheorem{corollary}[theorem]{Corollary}

\theoremstyle{definition}
\newtheorem{definition}[theorem]{Definition}
\newtheorem*{definition*}{Definition}
\newtheorem{example}[theorem]{Example}

\theoremstyle{remark}
\newtheorem{remark}{Remark}

\newtheorem{question}{Question}

\newcommand{\BUC}{\operatorname{BUC}}

\DeclareMathOperator{\tr}{tr}

\allowdisplaybreaks

\newcommand{\C}{\mathbb{C}}

\newcommand{\N}{\mathbb{N}}
\newcommand{\R}{\mathbb{R}}

\newcommand{\1}{\mathbbm{1}}

\newcommand{\Ac}{\mathcal{A}}
\newcommand{\Cc}{\mathcal{C}}
\newcommand{\Dc}{\mathcal{D}}
\newcommand{\Fc}{\mathcal{F}}
\newcommand{\Hc}{\mathcal{H}}
\newcommand{\Kc}{\mathcal{K}}
\newcommand{\Lc}{\mathcal{L}}
\newcommand{\Sc}{\mathcal{S}}
\newcommand{\Tc}{\mathcal{T}}

\newcommand{\Af}{\mathfrak{A}}

\DeclareMathOperator{\Real}{Re}
\renewcommand{\Re}{\Real}
\DeclareMathOperator{\Imag}{Im}
\renewcommand{\Im}{\Imag}

\DeclareMathOperator{\BO}{BO}
\DeclareMathOperator{\BDO}{BDO}
\DeclareMathOperator{\VO}{VO}
\DeclareMathOperator{\EssCom}{EssCom}

\DeclareMathOperator{\dist}{dist}
\DeclareMathOperator{\spann}{span}
\DeclareMathOperator{\supp}{supp}

\newcommand{\from}{\colon}

\newcommand{\scpr}[2]{\left\langle #1, #2 \right\rangle}
\renewcommand{\sp}{\scpr}
\newcommand{\abs}[1]{\left\lvert#1\right\rvert}
\newcommand{\norm}[1]{\left\lVert#1\right\rVert}
\newcommand{\set}[1]{\left\{ #1\right\}}

\renewcommand{\epsilon}{\varepsilon}

\begin{document}
\pagenumbering{gobble}
\title{Quantum harmonic analysis for polyanalytic Fock spaces}
\author{Robert Fulsche and Raffael Hagger}
\date{}
\maketitle
%\tableofcontents
\pagenumbering{arabic}

\begin{abstract}
    We develop the quantum harmonic analysis framework in the reducible setting and apply our findings to polyanalytic Fock spaces. In particular, we explain some phenomena observed in recent work by the second author and answer a few related open questions. For instance, we show that there exists a symbol such that the corresponding Toeplitz operator is unitary on the analytic Fock space but vanishes completely on one of the true polyanalytic Fock spaces. This follows directly from an explicit characterization of the kernel of the Toeplitz quantization, which we derive using quantum harmonic analysis. Moreover, we show that the Berezin transform is injective on the set of of Toeplitz operators. Finally, we provide several characterizations of the $\Cc_1$-algebra in terms of integral kernel estimates and essential commutants.
\end{abstract}

\section{Introduction}

The study of linear operators acting on reproducing kernel spaces is a recurring theme in operator theory. Among the big class of reproducing kernel Hilbert spaces, the Fock space (also called Segal--Bargmann--Fock space or any combination of these names) is one of the more popular ones to study; see \cite{Bauer_Isralowitz2012, Fulsche2020,hagger2021, XiaZheng, Zhu} for just a handful of the many works related to operators on the Fock space. An interesting variation of the standard Fock spaces is the class of \emph{polyanalytic Fock spaces}, which have received quite some attention in mathematical physics lately \cite{Abreu_Balazs_deGosson_Mouayn2015,Abreu_Groechenig_Romero2019,Behrndt_Holzmann_Lotoreichik_Raikov2022,Constales_Faustino_Krausshar2011,Haimi_Hedenmalm2013}, see also \cite{Balk1991} for a general account on polyanalytic functions. Their physical importance stems from the fact that the \emph{true} polyanalytic Fock spaces (to be defined below) are exactly the eigenspaces of the Landau operator $\mathcal L_{z, \overline{z}} := -\partial_z \partial_{\overline{z}} + \overline{z}\partial_{\overline{z}}$ acting on $L^2(\mathbb C, e^{-|z|^2})$, which has first been observed in \cite{Abreu2010b} and independently in \cite{Mouayn2011}. The true polyanalytic Fock spaces also appear very naturally in time-frequency analysis as the image of the Gabor transform with a Hermite window, whereas the full polyanalytic Fock spaces appear as the image of vectorial Gabor transforms with the window vector consisting of the first $n$ Hermite functions  \cite{Abreu2010,Abreu2010b,Keller_Luef2021,Luef_Skrettingland2021}. Both of these, apparently independent, features seemingly root on discoveries made in \cite{Grochenig_Lyubarskii2009}. Moreover, as for every reproducing kernel Hilbert space, Toeplitz operators are a natural class of operators to study on these spaces; see \cite{Arroyo_Sanchez_Hernandez_Lopez2021,Maximenko_Telleria2020,Rozenblum_Vasilevski2019} for some recent contributions. Even more recently, the second author investigated properties such as compactness and the Fredholm property of operators acting on these spaces including Toeplitz and Hankel operators \cite{hagger2023}. However, standard techniques such as the limit operator method were not quite strong enough to answer all relevant questions. For instance, in \cite[Theorem 23]{hagger2023} it was shown that if $f$ is a symbol such that the Toeplitz operator with symbol $f$ is compact on the analytic Fock space, it must be compact on every true polyanalytic Fock space. The second author then asked the natural question about the converse (see \cite[Question 31]{hagger2023}). The approach in \cite{hagger2023} was quite ad hoc and did not provide enough insight to construct a counterexample. This is where \emph{quantum harmonic analysis (QHA)} comes in. Writing Toeplitz operators as convolutions makes it quite obvious which symbols one has to choose to let it go wrong. Similarly, writing the Berezin transform as a convolution, makes it much easier to identify its properties (cf. \cite[Question 32]{hagger2023}). We provide some more details in the main results section below.

Quantum harmonic analysis, in the sense we understand it in this paper, goes back to the work of Werner \cite{werner84} in the early eighties. Quite surprisingly, it then took almost four decades until these methods were used to study operators on the Fock space \cite{Fulsche2020}. As it turns out, several deep theorems in this area can be proven in a much simpler way (and even extended) using QHA.  It is also worth mentioning that Werner's ideas have since been generalized to the affine group \cite{Berge_Berge_Luef_Skrettingland2022} and locally compact abelian groups \cite{Fulsche_Galke2023,Halvdansson2023}.

One of the main topics of classical harmonic analysis is the study of the convolution algebra $L^1(\mathbb R^d)$. In a similar vein, quantum harmonic analysis can be used to study the convolution algebra $L^1(\mathbb R^{2d}) \oplus \mathcal T^1(\mathcal H)$, where $\mathcal T^1(\mathcal H)$ is the ideal of trace class operators on a Hilbert space $\mathcal H$. \label{trace_class}\nomenclature{$\mathcal T^1(\mathcal H)$}{Trace class operators on a Hilbert space $\mathcal H$, p.~\pageref{trace_class}} Similarly, we will write $\mathcal T^2(\mathcal H)$\nomenclature{$\mathcal T^2(\mathcal H)$}{Hilbert--Schmidt operators on a Hilbert space $\mathcal H$, p.~\pageref{trace_class}}, $\mathcal K(\mathcal H)$\nomenclature{$\mathcal K(\mathcal H)$}{Compact operators on a Hilbert space $\mathcal H$, p.~\pageref{trace_class}}, $\mathcal L(\mathcal H)$\nomenclature{$\mathcal L(\mathcal H)$}{Bounded linear operators on a Hilbert space $\mathcal H$, p.~\pageref{trace_class}} and $\mathcal U(\mathcal H)$\nomenclature{$\mathcal U(\mathcal H)$}{Unitary operators on a Hilbert space $\mathcal H$, p.~\pageref{trace_class}} for the classes of Hilbert--Schmidt, compact, bounded and unitary operators on $\mathcal H$. Moreover, we will use $\mathcal T^1(\mathcal H_1, \mathcal H_2)$, $\mathcal K(\mathcal H_1, \mathcal H_2)$ and so on in case our operators act between two different Hilbert spaces $\Hc_1$ and $\Hc_2$.

In order to define a convolution on $L^1(\mathbb R^{2d}) \oplus \mathcal T^1(\mathcal H)$, which we shall recall later in detail, we rely on a family of unitary operators $W_z$, $z \in \mathbb R^{2d}$, satisfying the \emph{exponentiated canonical commutation relations (CCR)}:
\begin{align*}
    W_z W_w = e^{-i\sigma(z, w)/2}W_{z+w}.
\end{align*}
Here, $\sigma$ is a symplectic form on $\mathbb R^{2d}$. Now, if $\mathcal H$ is equal to the Fock space $F^2$ over $\C$, which consists of all entire functions on $\mathbb C$ that are square integrable with respect to the Gaussian measure $\mathrm{d}\mu(z) = \pi^{-1}\exp(-|z|^2)~\mathrm{d}z$, then such a family is given by the so-called \emph{Weyl operators}. For $z \in \C$ and $f \in F^2$ they are defined by
\[W_z f(w) = e^{w\overline{z} - \frac{|z|^2}{2}} f(w-z),\]
where we identified $\C \cong \R^2$ and used the symplectic form given by $\sigma(z,w) := 2\Imag(z\overline{w})$. In \cite{Fulsche2020}, the first author made use of the algebra structure that the Weyl operators induce on $L^1(\mathbb C) \oplus \mathcal T^1(F^2)$ to obtain a collection of results. Most critically, Werner's \emph{correspondence theorem} \cite{werner84} was used in these investigations to show, for example, that the Toeplitz algebra is linearly generated and equal to the $\Cc_1$-algebra, which will be defined below.

The main idea of the present paper is that the Weyl operators induce, in essentially the same way, a product structure on $L^1(\mathbb C) \oplus \mathcal T^1(\mathcal H)$, where $\mathcal H$ is now either one of the \emph{true polyanalytic Fock spaces} or one of the \emph{(full) polyanalytic Fock spaces}, two notions that we will recall in the main part of this paper. The true polyanalytic setting still very much encourages the use of QHA methods. Indeed, concrete objects such as \emph{Toeplitz operators} or the \emph{Berezin transform} can be written as a convolution of the symbol with a trace class operator just like on the analytic Fock space \cite{Fulsche2020}. But there is one crucial difference: When working on the analytic Fock space, the Toeplitz operator and the Berezin transform are given by convolution with a \emph{regular} operator (in the sense of Wiener's approximation theorem; see \cite{werner84} or Theorem \ref{thm:Correspondence_2} below), whereas on true polyanalytic Fock spaces the operator by which one has to convolve is no longer regular. This explains, from an QHA point of view, why the Toeplitz quantization $f \mapsto T_f$ and the Berezin transform $A \mapsto \widetilde{A}$ have non-trivial kernels in this setting. This has already been observed in \cite{Luef_Skrettingland2021}, for example. As we will show in this paper, these kernels can actually be computed explicitly.

For the full polyanalytic Fock spaces there is another significant difference in the quantum harmonic analysis. While the Weyl operators still satisfy the usual relations, they no longer act irreducibly on the Hilbert space. While this has only little influence on certain parts of the quantum harmonic analysis, the correspondence theory significantly changes. Therefore, we start our work with a general part on quantum harmonic analysis, which investigates the features of correspondence theory in the reducible case. This also means that we have to deal with operators between different Hilbert spaces. Luckily, this turns out to be straightforward, which allows us to skip some of the proofs.

\subsection{Main results}

We will start our discussion by an investigation of the correspondence theory in the case that the Weyl operators $W_z$ act reducibly on the Hilbert space $\Hc$. This will lead to several new results that generalize the correspondence theorem \cite[Theorem 4.1]{werner84}. These will then be used to study operators on the true polyanalytic Fock spaces $F_{(k)}^2$ and the (full) polyanalytic Fock spaces $F_n^2$. As our first main result we show that the kernel of the Toeplitz quantization on $L^{\infty}(\C)$ can be written as the weak$^*$ closure of
\[\spann\set{w \mapsto e^{2i\Imag(\xi\overline{w})} : \xi \in \Sigma_k},\]
where $\Sigma_k$ is a finite union of circles determined by the zeros of certain Laguerre polynomials (see Theorem \ref{theorem:toeplitz_quant} below). Similarly, the kernel of the Berezin transform on $\Lc(F^2_{(k)})$ can be written as the weak$^*$ closure of
\[\spann\set{W_{\xi} : \xi \in \Sigma_k}\]
as shown in Theorem \ref{thm:kernel_berezin_trafo}. Somewhat surprisingly, however, both the Toeplitz quantization and the Berezin transform remain injective if considered on $L^1(\mathbb C)$ or $\mathcal T^1(F_{(k)}^2)$, respectively. That is, the following statements hold (see Proposition \ref{prop:toeplitz_quant_wstar_dense} below):
\begin{itemize}
    \item[(i)] The map $L^1(\mathbb C) \ni f \mapsto T_{f, (k)} \in \Tc^1(F_{(k)}^2)$ is injective for every $k \geq 1$.
    \item[(ii)] The map $\mathcal T^1(F_{(k)}^2) \ni A \mapsto \widetilde{A} \in L^1(\C)$ is injective for every $k \geq 1$.
\end{itemize}
Of course, the dual statements also hold, that is, the ranges of $L^\infty(\Xi) \ni f \mapsto T_{f, (k)} \in \mathcal L(F_{(k)}^2)$ and $\mathcal L(F_{(k)}^2) \ni A \mapsto \widetilde{A} \in L^\infty(\mathbb C)$ are weak$^*$ dense. To extend this even further, we show that the Berezin transform is in fact injective on the set of Toeplitz operators with bounded symbols. This follows from a simple argument involving the Fourier transform and spectral synthesis.

Besides these facts, we also derive certain compactness results for operators acting on the polyanalytic Fock spaces, which are based on our general discussion of the correspondence theory. In particular, we answer \cite[Question 31]{hagger2023} in the negative and a variation of \cite[Question 32]{hagger2023} in the positive. That is, we show that for each $k \geq 2$ there exist bounded symbols $f$ such that the Toeplitz operator $T_{f,(k)}$ is compact but $T_{f,(1)}$ is not, and we show that $T_{f,(k)}$ is compact if and only if the Berezin transform vanishes at infinity. The latter is known for the generalized Berezin transform introduced in \cite{hagger2023} and, after observing that the usual Berezin transform is still injective on the Toeplitz operators, maybe not as surprising anymore.

As our final topic, we turn towards operator algebras on the polyanalytic Fock spaces. Similarly to the case of the standard Fock space, we consider the class of sufficiently and weakly localized operators, the algebra of band-dominated operators and the class $\mathcal C_1$ of operators on which the phase space shifts act continuously in operator norm. Our main result of this section is that these $C^\ast$ algebras of operators all coincide. Further, we show that this algebra can also be characterized as the essential commutant of the set of all Toeplitz operators with symbols of vanishing oscillation. Our findings in this part can be viewed as a generalization of results in \cite{hagger2021,Xia2015}, where, especially in the latter reference, different methods were used.

We conclude this introduction with a brief summary: After this introduction, we add a list with important notations used throughout the paper for the reader's convenience. Section 2 is dedicated to the discussion of general facts about quantum harmonic analysis, most notably the correspondence theory for the case when the Weyl system does not act irreducibly. Section 3 describes the applications of quantum harmonic analysis to the mapping properties of the Toeplitz quantization and the Berezin transform, as well as compactness properties of operators acting on the polyanalytic Fock spaces. In Section 4 we present our results about the previously mentioned operator algebras. Finally, Section 5 contains a short discussion including two open questions, which are related to our findings in Sections 3 and 4.

\printnomenclature

\section{Quantum harmonic analysis}\label{sec:qha}

In the following, we fix a finite dimensional symplectic vector space $(\Xi, \sigma)$ over $\R$. The real dimension of $\Xi$ will always be $2d$ in this section. Later on we will apply the concepts described here to the case $d = 1$, but considering the more general situation helps to understand our results in the context of quantum harmonic analysis and may also be useful for future reference.

We recall that a $\sigma$-representation $(\mathcal H, W)$ consists of an (infinite-dimensional) Hilbert space $\mathcal H$ together with a map
\[W \from \Xi \to \mathcal U(\mathcal H), \quad z \mapsto W_z,\]
which is assumed to be continuous with respect to the strong operator topology and satisfies the following composition formula:
\begin{equation} \label{eq:Weyl_operators}
W_z W_w = e^{-i\sigma(z, w)/2}W_{z+w}
\nomenclature{$W_z$}{Weyl operators, Eq.~\eqref{eq:Weyl_operators} on p.~\pageref{eq:Weyl_operators} and Eq.~\eqref{eq:Weyl_operators_concrete} on p.~\pageref{eq:Weyl_operators_concrete} for the concrete version on $L^2(\C,\mu)$}
\end{equation}
for $w,z \in \Xi$. The operators $W_z$ are usually referred to as \emph{Weyl operators}.

\subsection{Two $\sigma$-representations}\label{sec:2.1}

In \cite{werner84}, the concepts of QHA were developed with respect to one such $\sigma$-representation. Here, we will generalize these concepts to capture the interplay between two different $\sigma$-representations. Hence, we let $(\mathcal H_j, W^j)$ and $(\mathcal H_k, W^k)$ be two $\sigma$-representations. The notation of course already reveals that we will consider a (finite) family $(\mathcal H_j, W^j)_{j \in J}$ of $\sigma$-representations later on. Until further notice, we assume that the $\sigma$-representations are irreducible, that is, there is no proper subspace which reduces all the $W_z^j$ or $W_z^k$, respectively.

The possibly most important result on $\sigma$-representations is the Stone--von Neumann Theorem \cite[Theorem 1.50]{Folland1989}: For any two infinite-dimensional irreducible $\sigma$-representations $(\mathcal H_j, W^j)$ and $(\mathcal H_k, W^k)$ there exists a unitary operator \label{def:operator_Akj}\nomenclature{$\mathfrak A_{k, j}$}{Operator intertwining two irreducible $\sigma$-representations, p.~\pageref{def:operator_Akj}} $\mathfrak A_{k, j}: \mathcal H_j \to \mathcal H_k$ such that $\mathfrak A_{k,j} W_z^j = W_z^k \mathfrak A_{k,j}$ for every $z \in \Xi$. Further, this operator $\mathfrak A_{k,j}$ is unique up to a multiplicative constant of absolute value one. We always choose these constants such that $\mathfrak A_{k,j}^\ast = \mathfrak A_{j,k}$.
Due to this result, it seems pointless at first glance to consider two different $\sigma$-representations simultaneously. Nevertheless, for operator theoretical considerations it will turn out to be useful to follow this path. Our motivation for considering multiple $\sigma$-representations simultaneously arises from the study of polyanalytic Fock spaces, which will be explained later on.

For any (infinite-dimensional and irreducible) $\sigma$-representation $(\mathcal H, W)$, we can obtain a new $\sigma$-representation $(\mathcal H, V)$ by defining $V_z := W_{-z}$. Hence, by the Stone--von Neumann Theorem, there exists a unitary operator $U$ such that $W_{-z} = V_z = U W_z U^\ast$, that is, $W_{-z} U = U W_z$ for all $z \in \Xi$. Since the operator $U^2$ commutes with every $W_z$, Schur's lemma shows that $U^2$ is a constant multiple of the identity. Adjusting all the constants correctly, we can choose $U$ to be self-adjoint, and this $U$ is well-defined up to a multiplicative constant of $\pm 1$. For the two different $\sigma$-representations $(\mathcal H_j, W^j)$ and $(\mathcal H_k, W^k)$ we denote the respective operators by $U_j$ and $U_k$. Note that, as another immediate consequence of the Stone--von Neumann Theorem, $\mathfrak A_{k,j} U_j = \pm U_k \mathfrak A_{k,j}$. Upon choosing the constants $\pm 1$ in the definition of $U_j, U_k$ correctly, we can enforce $\mathfrak A_{k,j}U_j = U_k \mathfrak A_{k,j}$.

We now introduce the main objects of QHA. The standard $L^p$-spaces with respect to the Lebesgue measure on $\Xi$ will be denoted by $L^p(\Xi)$ and we will use $\|\cdot\|_p$ for the corresponding norm. For $f \in L^1(\Xi)$, we denote by
\begin{align}\label{eq:simpl_Fouriertrafo}
\nomenclature{$\mathcal F_\sigma(f)$}{Symplectic Fourier transform of $f$, Eq.~\eqref{eq:simpl_Fouriertrafo} on p.~\pageref{eq:simpl_Fouriertrafo}} \mathcal F_\sigma(f)(\xi) = c_{\sigma}\int_{\Xi} e^{-i\sigma(\xi, z)} f(z) \, \mathrm{d}z
\end{align}
the symplectic Fourier transform of $f$. Note that $c_\sigma$ is a non-negative constant which only depends on the symplectic form. $c_\sigma$ has to be arranged such that $\mathcal F_\sigma^2 = Id$. For the standard symplectic form $\sigma((x, \xi), (y, \eta)) = y\xi - x\eta$ on $\mathbb R^{2d}$ we have $c_\sigma = (2\pi)^{-d}$. For the symplectic form used in Section \ref{sec:3} and onward, we have $c_\sigma = \pi^{-1}$. A similar notion is the Fourier--Weyl transform for trace class operators on $\Hc_j$. Here, the integral is replaced by the trace and the characters are replaced by the Weyl operators, that is,
\begin{align}\label{eq:Fourier_Weyl_trafo}
\nomenclature{$\mathcal F_W(A)$}{Fourier-Weyl transform of $A$, Eq.~\eqref{eq:Fourier_Weyl_trafo} on p.~\pageref{eq:Fourier_Weyl_trafo}}\Fc_W(A)(\xi) := \tr(AW_{\xi}^j)
\end{align}
for any $A \in \Tc^1(\Hc_j)$.

Given $A \in \mathcal L(\mathcal H_j, \mathcal H_k)$ and $z \in \Xi$, we denote the \emph{shift of $A$ by $z$} by
\begin{align}\label{def:operator_shift}
\nomenclature{$\alpha_z(A), ~\alpha_z^{k,j}(A)$}{Shift of operator $A$, Eq.~\eqref{def:operator_shift} on p.~\pageref{def:operator_shift} and also p.~\pageref{def:operator_shift2}}\alpha_z(A) := \alpha_z^{k,j}(A) := W_z^{k} A W_{-z}^j.
\end{align}
Of course, we have $\| \alpha_z(A)\|_{op} = \| A\|_{op}$ for the operator norm $\|\cdot\|_{op}$. If $A \in \mathcal T^{1}(\mathcal H_j, \mathcal H_k)$, then we also have $\| \alpha_z(A)\|_{\mathcal T^1} = \| A\|_{\mathcal T^1}$ for the trace norm.

Further, we will denote
\begin{align}\label{def:parity_action}
\nomenclature{$\beta_-(A), ~\beta_-^{k,j}(A)$}{Adjoining $A$ by parity operator, Eq.~\eqref{def:parity_action} on p.~\pageref{def:parity_action}}\beta_{-}(A) := \beta_{-}^{k,j}(A) := U_k A U_j.
\end{align}
Note that this is independent of the factors $\pm 1$ which can be chosen in the construction of $U_j,U_k$. These actions are the operator analogues of the following standard actions on functions $f \in L^p(\Xi)$:
\begin{align}\label{def:shift_function}
\nomenclature{$\alpha_z(f)$}{Shift of $f$, Eq.~\eqref{def:shift_function} on p.~\pageref{def:shift_function}}\alpha_z(f)(w) := f(w-z), \quad \nomenclature{$\beta_-(f)$}{Parity operator applied to $f$, Eq.~\eqref{def:shift_function} on p.~\pageref{def:shift_function}} \beta_{-}(f)(w) = f(-w).
\end{align}

Since the maps $z \mapsto W_z^j$ and $z \mapsto W_z^k$ are continuous with respect to the strong operator topology, it is an easy exercise to prove that for $A \in \mathcal T^1(\mathcal H_j, \mathcal H_k)$, $z \mapsto \alpha_z(A)$ is a continuous map from $\Xi$ to $\mathcal T^1(\mathcal H_j, \mathcal H_k)$ (prove this for rank one operators first, then approximate). Hence, given $f \in L^1(\Xi)$ and $A \in \mathcal T^1(\mathcal H_j, \mathcal H_k)$, the expression
\begin{align}\label{def:conv_function_operator1}
\nomenclature{$f \ast A$}{Convolution of function and operator, Eqs.~\eqref{def:conv_function_operator1} and \eqref{def:conv_function_operator1} on p.~\pageref{def:conv_function_operator1}}f \ast A := A \ast f := \int_{\Xi} f(z) \alpha_z(A) \, \mathrm{d}z
\end{align} 
is well-defined as a Bochner integral in $\mathcal T^1(\mathcal H_j, \mathcal H_k)$, and is therefore contained in $\mathcal T^1(\mathcal H_j, \mathcal H_k)$ again. Moreover,
\begin{align} \label{eq:convolution_estimate_T^1}
    \norm{f \ast A}_{\Tc^1} \leq \norm{f}_1\norm{A}_{\Tc^1}
\end{align}
for all $f \in L^1(\Xi)$, $A \in \Tc^1(\Hc_j,\Hc_k)$. Similarly, if $f \in L^1(\Xi)$ and $A \in \Lc(\Hc_j,\Hc_k)$, we define
\begin{align}\label{def:conv_function_operator2}
f \ast A := \int_{\Xi} f(z)\alpha_z(A) \, \mathrm{d}z,
\end{align}
which is now an operator in $\Lc(\Hc_j,\Hc_k)$ and it is clear that $\norm{f \ast A}_{op} \leq \norm{f}_1\norm{A}_{op}$. It is also easy to see that $\mathfrak A_{j,k}(f \ast A) = f \ast (\mathfrak A_{j,k} A)$ and $(f \ast A)\mathfrak A_{j,k} = f \ast (A \mathfrak A_{j,k})$, where the latter are the usual convolutions of QHA with respect to only one $\sigma$-representation that are defined in the same way. 

Given $A \in \mathcal T^1(\mathcal H_j, \mathcal H_k)$ and $B \in \mathcal T^1(\mathcal H_j, \mathcal H_k)$, we define the convolution $A \ast B \from \Xi \to \C$ as
\begin{align}\label{def:conv_operators}
\nomenclature{$A \ast B$}{Convolution of operators, Eq.~\eqref{def:conv_operators} on p.~\pageref{def:conv_operators} and Eq.~\eqref{def:conv_operator2} on p.~\pageref{def:conv_operator2}}A \ast B(z) = \tr( A W_z^j U_j \mathfrak A_{j,k} B \mathfrak A_{j,k} U_k W_{-z}^k ).
\end{align}
Note that
\begin{align*}
A \ast B(z) &= \tr(A \mathfrak A_{j,k} \mathfrak A_{k,j} W_z^j U_j \mathfrak A_{j,k}B \mathfrak A_{j,k} U_k W_{-z}^k )\\
&= (A \mathfrak A_{j,k}) \ast (B\mathfrak A_{j,k})(z),
\end{align*}
where $A\mathfrak A_{j,k}, B\mathfrak A_{j,k} \in \mathcal T^1(\mathcal H_k)$ and $\| A\mathfrak A_{j,k}\|_{\mathcal T^1} = \| A\|_{\mathcal T^1}$, $\| B\mathfrak A_{j,k} \|_{\mathcal T^1} = \| B\|_{\mathcal T^1}$. Further, as the trace is invariant under cyclic permutations, we have $A \ast B = B \ast A$. Thus, we obtain from \cite[Lemma 3.1]{werner84}:

\begin{lemma}\label{lem:conv_estimate_1}
Let $A, B \in \mathcal T^1(\mathcal H_j, \mathcal H_k)$. Then $A \ast B \in L^1(\Xi)$ with
\begin{align*}
\| A \ast B\|_1 &\leq c_\sigma^{-1} \| A\|_{\mathcal T^1} \| B\|_{\mathcal T^1},\\
\int_{\Xi} A \ast B(z) \, \mathrm{d}z &= c_\sigma^{-1}\tr(A\mathfrak A_{j,k}) \tr(B\mathfrak A_{j,k}).
\end{align*}
\end{lemma}

Clearly, we can replace one of the factors of this convolution by a bounded operator, with the convolution defined by the same formula. We obtain:
\begin{lemma} \label{lem:convolution_estimate}
Let $A \in \mathcal T^1(\mathcal H_j, \mathcal H_k)$ and $B \in \mathcal L(\mathcal H_j, \mathcal H_k)$. Then $A \ast B \in L^\infty(\Xi)$ with
\begin{align*}
\| A \ast B\|_\infty \leq \| A\|_{op} \| B\|_{\mathcal T^1}.
\end{align*}
The analogous statement holds if $B$ is trace class and $A$ is just bounded.
\end{lemma}

We also want to define the convolution of $f \in L^\infty(\Xi)$ with $A \in \mathcal T^1(\mathcal H_j, \mathcal H_k)$ as an element of $\mathcal L(\mathcal H_j, \mathcal H_k)$. This is done weakly by considering $\mathcal L(\mathcal H_j, \mathcal H_k)$ as the dual space of $\mathcal T^1(\mathcal H_k, \mathcal H_j)$. More precisely, the operator $A \ast f \in \mathcal L(\mathcal H_j, \mathcal H_j)$ is defined by the relation
\begin{align*}
\tr((A \ast f) B ) = \int_{\Xi} f(z) (\beta_{-}(A) \ast B)(z) \, \mathrm{d}z, \quad B \in \mathcal T^1(\mathcal H_k, \mathcal H_j).
\end{align*}
By Lemma \ref{lem:conv_estimate_1}, it is clear that
\begin{align} \label{eq:convolution_estimate_L^inf}
    \norm{A \ast f}_{op} \leq c_\sigma^{-1} \norm{A}_{\Tc^1}\norm{f}_{\infty}
\end{align}
for $A \in \Tc^1$, $f \in L^{\infty}(\Xi)$. We will occasionally also use the notation $S \ast T$ for sets $S \subseteq \mathcal T^1(\mathcal H_j, \mathcal H_k)$ and $T \subseteq L^\infty(\Xi)$, by which we mean the set obtained from all possible convolutions of elements in $S$ and $T$: $S \ast T = \{ A \ast f: ~A \in S, f \in T\}$. Similarly, one defines the convolution of subsets of $\mathcal T^1(\mathcal H_j, \mathcal H_k)$ and $\mathcal L(\mathcal H_j, \mathcal H_k)$ and other combinations of sets, for which the convolution is well-defined.

All these convolutions interact with the shifts as expected, with proofs just as in the case of one $\sigma$-representation by direct verification: 
\begin{align}
\alpha_z(A \ast f) &= \alpha_z(A) \ast f = A \ast \alpha_z(f), \label{eq:shiftconv1}\\
\alpha_z(A \ast B) &= \alpha_z(A) \ast B = A \ast \alpha_z(B).\label{eq:shiftconv2}
\end{align}
In particular, if we define
\begin{align}\label{def:c1}
\nomenclature{$\mathcal C_1^{k,j}, ~\mathcal C_1(\mathcal H_j, \mathcal H_k)$}{Shift-continuous operators, Eq.~\eqref{def:c1} on p.~\pageref{def:c1}}\mathcal C_1^{k,j} := \mathcal C_1(\mathcal H_j, \mathcal H_k) := \{ A \in \mathcal L(\mathcal H_j, \mathcal H_k): ~\| \alpha_z(A) - A\|_{op} \to 0 \text{ as } |z| \to 0\},
\end{align}
we get that $A \ast f \in \mathcal C_1^{k,j}$ for $f \in L^\infty(\Xi)$ and $A \in \mathcal T^1(\mathcal H_j, \mathcal H_k)$ or $f \in L^1(\Xi)$ and $A \in \mathcal L(\mathcal H_j, \mathcal H_k)$. Let us emphasize the following: Since the convolution is commutative and associative, there seems to be no natural order in which we write the factors of the convolution. When one of the factors is not in $L^1$ or $\mathcal T^1$, we will try to adopt the convention that the convolution is formally a map $\ast: (L^1(\Xi) \oplus \mathcal T^1(\mathcal H)) \times (L^\infty(\Xi) \oplus \mathcal L(\mathcal H)) \to (L^\infty(\Xi) \oplus \mathcal L(\mathcal H))$, where the parts that are not needed are usually omitted.

Here are some of the main facts of QHA for the case of two representations:

\begin{theorem} \label{thm:Correspondence_2}
For $A \in \mathcal T^1(\mathcal H_j, \mathcal H_k)$ the following statements are equivalent:
\begin{enumerate}
\item The map $L^1(\Xi) \ni f \mapsto A \ast f \in \mathcal T^1(\mathcal H_j, \mathcal H_k)$ has dense range.
\item The map $\mathcal T^1(\mathcal H_j, \mathcal H_k) \ni B \mapsto A \ast B \in L^1(\Xi)$ has dense range.
\item The map $\mathcal L(\mathcal H_j, H_k) \ni B \mapsto A \ast B \in L^\infty(\Xi)$ is injective.
\item The map $L^\infty(\Xi) \ni f \mapsto A \ast f \in \mathcal L(\mathcal H_j, \mathcal H_k)$ is injective.
\item $A\mathfrak A_{j,k}$ is a regular operator in the sense of \cite{werner84}, i.e. $\mathcal F_W(A\mathfrak A_{j,k})(\xi) = \tr(A\mathfrak{A}_{j,k} W_\xi^k) \neq 0$ for all $\xi \in \Xi$.
\item $\mathfrak A_{j,k}A$ is a regular operator in the sense of \cite{werner84}, i.e. $\mathcal F_W(\mathfrak A_{j,k}A)(\xi) = \tr(\mathfrak A_{j,k} A W_\xi^j) \neq 0$ for all $\xi\in \Xi$.
\item $A \ast A$ is a regular function, i.e. $\mathcal F_\sigma(A \ast A)(\xi) \neq 0$ for every $\xi \in \Xi$.
\item The linear span of $\set{\alpha_z(A) : z \in \Xi}$ is dense in $\mathcal T^1(\mathcal H_j, \mathcal H_k)$.
\end{enumerate}
\end{theorem}

\begin{proof}
Equivalence of these statement follows easily from some standard duality arguments, the identities $\mathfrak A_{j,k} (f \ast A) = f \ast (\mathfrak A_{j,k}A)$, $(f \ast A) \mathfrak A_{j,k} = f \ast (A\mathfrak A_{j,k})$, $A \ast B = (A\mathfrak A_{j,k}) \ast (B\mathfrak A_{j,k})$ and the results of \cite{werner84}.
\end{proof}

An operator $A \in \mathcal T^1(\mathcal H_j, \mathcal H_k)$ satisfying the above equivalent properties is of course also called \emph{regular}. Note that $A$ is regular if any only if $A^\ast \in \mathcal T^1(\mathcal H_k, \mathcal H_j)$ is regular. Moreover, we observe that regular operators always exist. This can be seen as follows for $\Hc_j = \Hc_k$ (for the general case just multiply with $\Af_{k,j}$). Choose $f \from \Xi \to \R$, $f(z) := e^{-\frac{1}{2}|z|^2}$, where $\abs{\cdot}$ denotes the Euclidean norm on $\Xi \cong \R^{2d}$, and $A := \Fc_W^{-1}(f)$. Then $A^2 \in \Tc^1(\Hc_j)$ is trace class and $\Fc_W(A^2)$ is equal to the twisted convolution $f \ast_{\sigma} f$ (cf.\ \cite[p. 26]{Folland1989} or \cite[Corollary 5.21]{Fulsche_Galke2023}). A straightforward computation then shows that $f \ast_{\sigma} f$ does not vanish on $\Xi$.

Here is another important fact, where we denote by $g_t$ any positive $\delta$-sequence in $L^1(\Xi)$ for $t \to 0$. We will usually make the choice\footnote{Here, $|z|^2$ has to be interpreted as $|z|^2 = \sigma(z, Jz)$, where $J$ is a fixed complex structure on $\Xi$ such that $\sigma(Jz,Jw) = \sigma(z,w)$ and $\sigma(z,Jz) > 0$ for $z \neq 0$.}
\begin{align*}
g_t(z) = \frac{1}{(\pi t)^d} e^{-\frac{|z|^2}{t}},
\end{align*} 
but this is not necessary.
\begin{lemma}\label{lem:approximation}
Let $A \in \mathcal L(\mathcal H_j, \mathcal H_k)$. Then, $A \in \mathcal C_1(\mathcal H_j, \mathcal H_k)$ if any only if $g_t \ast A \to A$ in operator norm as $t \to 0$.
\end{lemma}
\begin{proof}
Note that $A \in \mathcal C_1(\mathcal H_j, \mathcal H_k)$ if and only if $A\mathfrak A_{j,k} \in \mathcal C_1(\mathcal H_k, \mathcal H_k)$. Now, one derives the lemma from the analogous statement for $\mathcal C_1(\mathcal H_k, \mathcal H_k)$, \cite[Prop.\ 2.16]{Fulsche2020}, and the identity $(g_t \ast A)\mathfrak A_{j,k} = g_t \ast (A \mathfrak A_{j,k})$. 
\end{proof}
Here is the $\mathcal H_j - \mathcal H_k$ version of the correspondence theorem. The algebra of bounded, uniformly continuous functions $f \from \Xi \to \C$ is henceforth denoted by $\BUC(\Xi)$. The proof of the theorem is, up to some straightforward modifications, the same as for the case $\mathcal H_j = \mathcal H_k$ presented in \cite{werner84}.

\begin{theorem}\label{thm:corr_1}
Let $A \in \mathcal T^1(\mathcal H_j, \mathcal H_k)$ be regular. For any closed, $\alpha$-invariant subspace $\mathcal D_1$ of $\mathcal C_1^{k,j}$ there is a unique closed, $\alpha$-invariant subspace $\mathcal D_0$ of $\operatorname{BUC}(\Xi)$ such that the following holds true: Given $B \in \mathcal C_1^{k,j}$, it holds:
\begin{align*}
B \in \mathcal D_1 \Leftrightarrow A \ast B \in \mathcal D_0.
\end{align*}
The corresponding spaces satisfy
\begin{align*}
\mathcal D_0 = \overline{\mathcal T^1(\mathcal H_j, \mathcal H_k) \ast \mathcal D_1}, \quad \mathcal D_1 = \overline{\mathcal T^1(\mathcal H_j, \mathcal H_k) \ast \mathcal D_0}.
\end{align*}
\end{theorem}
\begin{remark}
    Let $A \in \mathcal T^1(\mathcal H_j, \mathcal H_k)$ be regular.  As a direct consequence of Theorem \ref{thm:Correspondence_2}, the corresponding spaces are equally well given by $\mathcal D_0 = \overline{A \ast \mathcal D_1}$ and $\mathcal D_1 = \overline{A \ast \mathcal D_0}$. Of course, by the previous theorem, these relations are independent of the particular choice of the regular operator $A$.
\end{remark}

The standard correspondences are established as in \cite{werner84}. In particular, we obtain the correspondences $\operatorname{BUC}(\Xi) \leftrightarrow \mathcal C_1^{k,j}$, $C_0(\Xi) \leftrightarrow \mathcal K(\mathcal H_j, \mathcal H_k)$, where $C_0(\Xi) \subseteq \BUC(\Xi)$ denotes the ideal of functions vanishing at infinity. Due to its significance for our paper, we give the quick proof of the latter correspondence (see also \cite{werner84}).

\begin{lemma}\label{lem:correspondence_compact}
If $\mathcal D_0 = C_0(\Xi)$, then $\mathcal D_1 = \mathcal K(\mathcal H_j, \mathcal H_k)$.
\end{lemma}

\begin{proof}
We know that $\mathcal D_1 = \overline{\mathcal T^1(\mathcal H_j, \mathcal H_k) \ast C_0(\Xi)}$. Since $L^1(\Xi) \cap C_0(\Xi)$ is dense in $L^1(\Xi)$, and using Eq.~\eqref{eq:convolution_estimate_T^1}, we obtain that $\mathcal T^1(\mathcal H_j, \mathcal H_k) \ast (L^1(\Xi) \cap C_0(\Xi))$ is dense in $\mathcal T^1(\mathcal H_j, \mathcal H_k) \ast L^1(\Xi) = \mathcal T^1(\mathcal H_j, \mathcal H_k)$ with respect to the trace norm. Taking the closure in operator norm, we obtain that $\mathcal D_1 \supseteq \mathcal K(\mathcal H_j, \mathcal H_k)$.

On the other hand, $\mathcal T^1(\mathcal H_j, \mathcal H_k) \ast (L^1(\Xi) \cap C_0(\Xi)) \subseteq \mathcal T^1(\mathcal H_j, \mathcal H_k)$. Since $L^1(\Xi) \cap C_0(\Xi)$ is also dense in $C_0(\Xi)$, and using Eq.~\eqref{eq:convolution_estimate_L^inf}, we obtain that $\mathcal T^1(\mathcal H_j, \mathcal H_k) \ast (L^1(\Xi) \cap C_0(\Xi))$ is dense in $\mathcal D_1$ with respect to the operator norm. Taking the closure, we obtain that $\mathcal D_1 = \mathcal K(\mathcal H_j, \mathcal H_k)$.
\end{proof}

In our investigations of operators on polyanalytic Fock spaces, we will also have to deal with operators which are not regular but merely $\infty$-regular. An operator $A \in \mathcal T^1(\mathcal H_j)$ is called $\infty$-regular (cf.\ \cite{Kiukas_Lahti_Schultz_Werner2012}) if the set $\{ \xi \in \Xi: \tr(AW_\xi^j) = 0\}$ has dense complement. Similarly, one may consider the notion of $\infty-$regularity for $A \in \mathcal T^1(\mathcal H_j, \mathcal H_k)$ by considering the set of zeros of $\tr(A\mathfrak A_{j,k}W_\xi^k)$. While $\infty$-regular operators do not give rise to the full correspondence theory anymore, they still have some related properties. Again, the proof is essentially the same as for $\mathcal H_j = \mathcal H_k$, which was given in \cite{Kiukas_Lahti_Schultz_Werner2012}.
\begin{theorem}[{\cite[Propositions 3 and 4]{Kiukas_Lahti_Schultz_Werner2012}}]\label{thm:inftyregular}
    For $A \in \mathcal T^1(\mathcal H_j,\mathcal H_k)$ the following statements are equivalent:
    \begin{enumerate}
        \item $A$ is $\infty$-regular.
        \item $\{ \alpha_z(A): ~z \in \Xi\}$ spans a weak$^\ast$ dense subspace of $\mathcal L(\mathcal H_j,\mathcal H_k)$.
        \item $L^1(\Xi) \ni f \mapsto f \ast A \in \mathcal T^1(\mathcal H_j,\mathcal H_k)$ is injective.
        \item $\mathcal T^1(\mathcal H_j,\mathcal H_k) \ni B \mapsto A \ast B \in L^1(\Xi)$ is injective.
        \item $L^\infty(\Xi) \ni f \mapsto A \ast f \in \mathcal L(\mathcal H_j,\mathcal H_k)$ has weak$^\ast$ dense range.
        \item $\mathcal L(\mathcal H_j,\mathcal H_k) \ni B \mapsto A \ast B \in L^\infty(\Xi)$ has weak$^\ast$ dense range.
        \item $A \ast C_0(\Xi) = \{ A \ast f: ~f\in C_0(\Xi)\}$ is dense in $\mathcal K(\mathcal H_j,\mathcal H_k)$.
        \item $A \ast \mathcal K(\mathcal H_j,\mathcal H_k) = \{ A \ast B: ~B \in \mathcal K(\mathcal H_j,\mathcal H_k)\}$ is dense in $C_0(\Xi)$.
    \end{enumerate}
\end{theorem}

\subsection{Sums of irreducible $\sigma$-representations}\label{sec:reducible}
In this section we assume that $(\mathcal H, W)$ is a $\sigma$-representation, which is no longer irreducible, but the sum of finitely many irreducible $\sigma$-representations. More precisely, we assume
\begin{align*}
\mathcal H = \bigoplus\limits_{j=1}^N \mathcal H_j
\end{align*}
for closed subspaces $\Hc_j \subseteq \Hc$ that are invariant under every $W_z$, $z \in \Xi$. Regarding the operators $U_j$ and $\Af_{j,k}$, we adopt the conventions from the previous section in the obvious way. We emphasize that it is possible to choose the constants from the discussions at the beginning of Section \ref{sec:2.1} consistently. Let us first recall that we asked for the conditions $\mathfrak A_{k, j}^\ast = \mathfrak A_{j, k}$, $U_j$ to be self-adjoint and $\mathfrak A_{k, j} U_j = U_k \mathfrak A_{k, j}$. These properties can be enforced by fixing the operators in the correct order: First, pick $U_1$ self-adjoint. Then, fix some choice of $\mathfrak A_{2, 1}$. From these two, we fix the choice of $U_2$ such that $\mathfrak A_{2, 1} U_1 = U_2 \mathfrak A_{2, 1}$. Next, we fix some choice of $\mathfrak A_{3,2}$ and then $U_3$ such that $\mathfrak A_{3,2} U_2 = U_3 \mathfrak A_{3,2}$. We continue this scheme until we have all $\mathfrak A_{k+1, k}$ and $U_k$ fixed. Now, for $j < k$ we let $\mathfrak A_{k, j} = \mathfrak A_{k, k-1} \mathfrak A_{k-1, k-2} \dots \mathfrak A_{j+1, j}$. Finally, for $j > k$ we let $\mathfrak A_{k, j} = \mathfrak A_{j, k}^\ast$. With these choices, all the conventions in Section \ref{sec:2.1} are satisfied.

We will often write operators $A \in \mathcal L(\mathcal H)$ in their matrix representation \label{def:matrix_repr}\nomenclature{$(A_{k,j})_{j,k=1, \dots, N}, A_{k,j}$}{Matrix representation and matrix entries of operator $A$, p.~\pageref{def:matrix_repr}}$(A_{k,j})_{j,k=1, \dots, N}$, where $A_{k,j} = P_{(k)} A|_{\mathcal H_j} \in \mathcal L(\mathcal H_j,\mathcal H_k)$. Here, \label{def:Pk}\nomenclature{$P_{(k)}$}{Orthogonal projection on subspace $\mathcal H_k$ of $\mathcal H$, p.~\pageref{def:Pk}}$P_{(k)}$ is the orthogonal projection onto $\mathcal H_k$. As in the irreducible case, we can define \label{def:operator_shift2}$\alpha_z(A) := W_z A W_{-z}$, which implies $\alpha_z(A)_{k,j} = \alpha_z(A_{k,j})$ for all $j,k = 1,\ldots,N$. We further define 
\begin{align}\label{def:C1_2}
\nomenclature{$\mathcal C_1, ~\mathcal C_1(\mathcal H)$}{Shift-continuous operators, Eq.~\eqref{def:C1_2} on p.~\pageref{def:C1_2}}\mathcal C_1 := \mathcal C_1(\mathcal H) := \{ A \in \mathcal L(\mathcal H): \| \alpha_z(A) - A\|_{op} \to 0 \text{ as } |z| \to 0\}.
\end{align}
Since
\begin{align*}
\| A_{k,j}\| \leq \| A\| \leq N\max_{ j, k = 1, \dots, N} \| A_{k,j}\|,
\end{align*}
we obtain that $A \in \mathcal C_1$ if and only if $A_{k,j} \in \mathcal C_1^{k,j}$ for all $j,k = 1,\ldots,N$. Moreover, $A \in \mathcal T^1(\mathcal H)$ if and only if $A_{k,j} \in \mathcal T^1(\mathcal H_j, \mathcal H_k)$ for all $j, k$. Given $A \in \mathcal T^1(\mathcal H)$ and $B \in \mathcal L(\mathcal H)$, we set
\begin{align}\label{def:conv_operator2}
A \ast B(z) := (A_{k,j} \ast B_{k,j}(z))_{j,k = 1, \dots, N},
\end{align}
which is an $N \times N$ matrix with entries in $\operatorname{BUC}(\Xi)$. If $B$ is also a trace class operator, the matrix entries are additionally in $L^1(\Xi)$. Further, given $F = (F_{k,j})_{j,k = 1, \dots, N} \in L^\infty(\Xi)^{N \times N}$ and $A \in \mathcal T^1(\mathcal H)$, then we define
\begin{align}\label{def:conv_operator_function_matrix}
\nomenclature{$(A \ast F)_{k,j}$}{Convolution of operator and matrix-valued function, Eq.~\eqref{def:conv_operator_function_matrix} on p.~\pageref{def:conv_operator_function_matrix}}(A \ast F)_{k,j} := A_{k,j} \ast F_{k,j}.
\end{align}
Analogously we define $F \ast A$ for $F_{k,j} \in L^1(\Xi)$ and $A \in \mathcal L(\mathcal H)$. To stay in the same formalism with the convolution of functions-valued matrices, we define $F \ast G \in L^p(\Xi)^{N \times N}$ by
\begin{align}\label{def:conv_function_function_matrix}
    \nomenclature{$(F \ast G)_{k,j}$}{Convolution of matrix-valued functions, Eq.~\eqref{def:conv_function_function_matrix} on p.~\pageref{def:conv_function_function_matrix}}(F \ast G)_{k,j} = F_{k,j} \ast G_{k,j}
\end{align}
for $F\in L^1(\Xi)^{N \times N}$ and $G \in L^p(\Xi)^{N \times N}$, $1 \leq p \leq \infty$.

Within this framework of a finite sum of irreducible $\sigma$-representations, we are interested in the analogous results to the Correspondence Theorem \ref{thm:corr_1}. The first idea one could come up with is an entrywise version of this theorem, which would read as follows:

\begin{theorem}\label{cor:maxpart}
 For each $j, k = 1, \dots, N$ let $\mathcal D_1^{k,j} \subseteq \mathcal C_1^{k, j}$ be a closed, $\alpha$-invariant subspace. Consider
    \begin{align*}
        \mathcal D_1 := \{ B = (B_{k,j})_{j, k = 1, \dots, N} \in \mathcal C_1(\mathcal H): B_{k,j} \in \mathcal D_1^{k,j}\} \cong \bigoplus_{j,k = 1, \dots, N} \mathcal D_1^{k,j},
    \end{align*}
    which is an $\alpha$-invariant closed subspace of $\mathcal C_1(\mathcal H)$. Then, there are unique $\alpha$-invariant closed subspaces $\mathcal D_0^{k,j} \subseteq \BUC(\Xi)$, $j, k = 1, \dots, N$, such that with 
    \begin{align*}
        \mathcal D_0 := \{ F = (F_{k,j})_{j,k = 1,\dots, N} \in \BUC(\Xi)^{N \times N}: F_{k,j} \in \mathcal D_0^{k,j}\} \cong \bigoplus_{j,k = 1, \dots, N} \mathcal D_0^{k,j}
    \end{align*}
    the following holds true: Given $B \in \mathcal C_1(\mathcal H)$ and $A \in \mathcal T^1(\mathcal H)$ such that every matrix entry $A_{k,j}$ is regular, it holds:
    \begin{align*}
        B \in \mathcal D_1 \Leftrightarrow A \ast B \in \mathcal D_0.
    \end{align*}
    Moreover,
    \begin{align*}
       \mathcal D_0 = \overline{A \ast \mathcal D_1} = \overline{\mathcal T^1(\mathcal H) \ast \mathcal D_1} \quad \text{and} \quad \mathcal D_1 = \overline{A \ast \mathcal D_0} = \overline{\mathcal T^1(\mathcal H) \ast \mathcal D_0}, 
    \end{align*}
    which is equivalent to
    \begin{align*}
        \mathcal D_0^{k,j} = \overline{A_{k,j} \ast \mathcal D_1^{k, j}} = \overline{\mathcal T^1(\mathcal H_j, \mathcal H_k) \ast \mathcal D_1^{k, j}} \quad \text{and} \quad \mathcal D_1^{k,j} = \overline{A_{k,j} \ast \mathcal D_0^{k, j}} = \overline{\mathcal T^1(\mathcal H_j, \mathcal H_k) \ast \mathcal D_0^{k, j}}
    \end{align*}
    for all $j,k = 1,\ldots,N$. 
\end{theorem}

\begin{proof}
The statement follows immediately from entrywise applications of Theorem \ref{thm:corr_1}.
\end{proof}

\begin{remark}
Let us elaborate on the notation $\mathcal D_1 \cong \bigoplus\limits_{j, k = 1, \dots, N} \mathcal D_1^{k, j}$ in the formulation of the previous theorem, which might seem odd. It should be understood as an isomorphism of topological vector spaces, which is due to the following: The left-hand side comes endowed with the operator norm of $\mathcal L(\mathcal H)$, while the right-hand side is naturally endowed with the supremum of the operator norms of $\mathcal L(\mathcal H_j, \mathcal H_k)$. While these are certainly different norms, this is also no issue: The key property of the space is its closedness in $\mathcal C_1(\mathcal H)$, which is only a property of the underlying linear topology. Since both norms in question are equivalent, they of course induce the same topology. The same comment also applies to $\mathcal D_0 \cong \bigoplus\limits_{j,k = 1, \dots, N} \mathcal D_0^{k,j}$, of course.
\end{remark}

Later, we will need the following corollary to the previous result.

\begin{corollary}\label{cor:equalspaces}
Let $\mathcal D_0^{k, j}, \mathcal D_1^{k, j}$ be as in the previous theorem. If $\mathcal D_0^{k, j} = \mathcal D_0^{k', j'}$ for some tuples $(j, k), (j', k')$ with $1 \leq j, j', k, k' \leq N$, then $\mathcal D_1^{k, j} = \mathfrak A_{k,k'}\mathcal D_1^{k', j'} \mathfrak A_{j', j}$. 
\end{corollary}

\begin{proof}
Choose a regular $A_{k,j} \in \Tc^1(\Hc_j,\Hc_k)$. Then $A_{k',j'} := \mathfrak A_{k',k} A_{k,j} \mathfrak A_{j,j'}$ is also regular and we get
\[\mathcal D_1^{k,j} = \overline{A_{k,j} \ast \mathcal D_0^{k,j}} = \overline{(\mathfrak A_{k,k'} A_{k',j'} \mathfrak A_{j',j}) \ast \mathcal D_0^{k',j'}} = \mathfrak A_{k,k'} \overline{A_{k',j'} \ast \mathcal D_0^{k',j'}} \mathfrak A_{j',j} = \mathfrak A_{k,k'} \mathcal D_1^{k',j'} \mathfrak A_{j',j}.\qedhere\]
\end{proof}
Especially when considering Toeplitz operators on polyanalytic Fock spaces, which we do in Section \ref{sec:3} below, it turns out that this is not the kind of spaces one would like to consider. Indeed, in this case, all of the $F_{k,j}$ are usually the same. Even though the following result is not quite what we need for Toeplitz operators later on, such assumptions directly lead to another version of the correspondence theorem. For a refined version that is more useful for Toeplitz operators we refer to Theorem \ref{thm:refined_correspondence_theorem} below.

\begin{theorem}\label{cor:minpart}
Let $\mathcal D_1^{1,1} \subseteq \mathcal C_1^{1,1}$ be $\alpha$-invariant and closed. Consider 
\begin{align*}
    \mathcal D_1 &:= \{ B = (B_{k,j})_{j, k = 1, \dots, N} \in \mathcal C_1(\mathcal H): B_{1,1} \in \mathcal D_1^{1,1}, \\
    &\quad \quad B_{k,j} = \mathfrak A_{k,1} B_{1,1} \mathfrak A_{1,j} \text{ for all } j, k = 1, \dots, N\}.
\end{align*}
Then there exists a unique $\alpha$-invariant closed subspace $\mathcal D_0^{1,1} \subseteq \BUC(\Xi)$ such that with
\begin{align*}
    \mathcal D_0 &:= \{ F = (F_{k,j})_{j, k = 1, \dots, N} \in \BUC(\Xi)^{N \times N}: F_{1,1} \in \mathcal D_0^{1,1}, \\
    &\quad \quad F_{k,j} = F_{1,1} \text{ for all } j, k = 1, \dots N\}
\end{align*}
the following holds: If $B \in \mathcal C_1(\mathcal H)$ and $A \in \mathcal T^1(\mathcal H)$ such that $A_{1,1}$ is regular and $A_{k,j} = \mathfrak A_{k,1} A_{1,1} \mathfrak A_{1,j}$ for all $j, k = 1, \dots N$, then
\begin{align} \label{eq:minpart}
    B \in \mathcal D_1 \Leftrightarrow A \ast B \in \mathcal D_0.
\end{align}
The spaces $\mathcal D_0$ and $\mathcal D_0^{1,1}$ satisfy
\begin{align*}
    \mathcal D_0 = \overline{\{ C \in \mathcal T^1(\mathcal H):~ C_{k,j} = \mathfrak A_{k,1} C_{1,1} \mathfrak A_{1,j} \text{ for all } j, k = 1, \dots, N\} \ast \mathcal D_1},
\end{align*}
and $\mathcal D_0^{1,1} = \overline{\mathcal T^1(\mathcal H_1) \ast \mathcal D_1^{1,1}}$, respectively.
\end{theorem}

\begin{proof}
We first show \eqref{eq:minpart} in case $\mathcal D_0^{1,1} = \overline{\mathcal T^1(\mathcal H_1) \ast \mathcal D_1^{1,1}}$, which proves the existence of such a corresponding space. Uniqueness will be proven at the end.

The implication ``$B \in \mathcal D_1 \Rightarrow A \ast B \in \mathcal D_0$'' follows directly from the definition of $\mathcal D_0$ and $\mathcal D_1$. On the other hand, for $B \in \mathcal C_1(\mathcal H)$, we have
\begin{align*}
    (A \ast B)_{k,j} = A_{k,j} \mathfrak A_{j,k} \ast B_{k,j} \mathfrak A_{j,k} = (\mathfrak A_{j,1} A_{1,1} \mathfrak A_{1,k}) \ast (B_{k,j} \mathfrak A_{j,k}) = A_{1,1} \ast (\mathfrak A_{1,k} B_{k,j} \mathfrak A_{j,1}).
\end{align*}
Hence, $A \ast B \in \mathcal D_0$ implies $A_{1,1} \ast (\mathfrak A_{1,k} B_{k,j} \mathfrak A_{j,1}) = A_{1,1} \ast B_{1,1} \in \mathcal D_0^{1,1}$ for every pair $(j,k)$. Since $A_{1,1}$ is regular, this shows $\mathfrak A_{1,k} B_{k,j} \mathfrak A_{j,1} = B_{1,1}$. Now, using the Correspondence Theorem \ref{thm:corr_1} in case $j = k = 1$, we get $B \in \mathcal D_1$ by observing that $B_{1,1} \in \mathcal D_1^{1,1}$ if and only if $A_{1,1} \ast B_{1,1} \in \mathcal D_0^{1,1}$. This shows \eqref{eq:minpart}.

Let
\begin{align*}
    \mathcal T^1(\mathcal H)_{M_{min}} := \{ C \in \mathcal T^1(\mathcal H): ~C_{k,j} = \mathfrak A_{k,1} C_{1,1} \mathfrak A_{1,j} \text{ for all } j, k = 1, \dots, N\}.
\end{align*}
The subscript notation we use here will be explained later on. Since $A \ast B \in \mathcal D_0$ whenever $B \in \mathcal D_1$, we clearly have $\mathcal D_0 \supseteq \overline{A \ast \mathcal D_1}$. By Theorem \ref{thm:Correspondence_2} and the $\alpha$-invariance of $\mathcal D_1^{1,1}$, we obtain $\mathcal D_0 \supseteq \overline{\mathcal T^1(\mathcal H)_{M_{min}} \ast \mathcal D_1}$. On the other hand, if $F \in \Dc_0$, then $F_{1,1} \in \Dc_0^{1,1} = \overline{\mathcal T^1(\mathcal H_1) \ast \mathcal D_1^{1,1}}$ and $F_{k,j} = F_{1,1}$ for all $j,k = 1,\ldots,N$. It follows $F \in \overline{\mathcal T^1(\mathcal H)_{M_{min}} \ast \mathcal D_1}$ and hence
\[\Dc_0 = \overline{\mathcal T^1(\mathcal H)_{M_{min}} \ast \mathcal D_1}.\]

To show uniqueness, assume that $\tilde{\Dc}_0$ is also an $\alpha$-invariant closed subspace that satisfies the correspondence property, that is, $B \in \mathcal D_1$ if and only if $A \ast B \in \tilde{\mathcal D}_0$. By the same argument as above, $\tilde{\Dc}_0$ must contain $\overline{\mathcal T^1(\mathcal H)_{M_{min}} \ast \mathcal D_1} = \mathcal D_0$. So let $F \in \tilde{\mathcal D}_0$. Then also $A \ast A \ast F \in \tilde{\mathcal D}_0$ because $\tilde{\Dc}_0$ is $\alpha$-invariant and closed. It follows that $A \ast F \in \mathcal D_1$, which shows $A \ast A \ast F \in \overline{\mathcal T^1(\mathcal H)_{M_{min}} \ast \mathcal D_1}$. But, since $A \ast A$ is a matrix the entries of which are identical regular functions in $L^1(\Xi)$, Wiener's approximation theorem implies that $F$ can be approximated by matrix-valued functions of the form $\sum\limits_{\ell} c_\ell \alpha_{z_\ell}(A \ast A \ast F)$ with $c_{\ell} \in \C$, $z_{\ell} \in \Xi$. Since $\overline{\mathcal T^1(\mathcal H)_{M_{min}} \ast \mathcal D_1}$ is $\alpha$-invariant and closed, this yields $F \in \overline{\mathcal T^1(\mathcal H)_{M_{min}} \ast \mathcal D_1}$. Hence $\tilde{\Dc}_0 = \Dc_0$, which also implies that $\Dc_0^{1,1}$ must be unique.
\end{proof}

Both Theorem \ref{cor:maxpart} and Theorem \ref{cor:minpart} are instances of a more general result. In order to formalize this, we introduce and fix a partition $M$ of $\{ (k, j): ~1 \leq j, k \leq N\}$ in the following. There are two particular partitions that we will sometimes encounter (or have already encountered in the proof of the previous theorem), the minimal one $M_{min}$ with $|M_{min}| = 1$ and the maximal one $M_{max}$ with $|M_{max}| = N^2$. We define
\begin{align*}
    \mathcal T^1(\mathcal H)_M = \{ A \in \mathcal T^1(\mathcal H):\ A_{k,j} = \mathfrak A_{k, k'} A_{k', j'}\mathfrak A_{j', j} \text{ for all } (k, j), (k', j') \in m, m \in M\}.
\end{align*}
Given an operator $A \in \mathcal L(\mathcal H)$ and $m \in M$, we define $A_m  \in \mathcal L(\mathcal H)$ by:
\begin{align*}
    (A_m)_{k, j} = \begin{cases}
    A_{k, j}, \quad &(k, j) \in m,\\
    0, \quad &(k, j) \not \in m.
    \end{cases}
\end{align*}
Similarly, for $F \in L^1(\Xi)^{N \times N}$ or $F \in L^\infty(\Xi)^{N \times N}$, we define $F_m$ by:
\begin{align*}
    (F_m)_{k, j} = \begin{cases}
    F_{k, j}, \quad &(k, j) \in m,\\
    0, \quad &(k, j) \not \in m.
    \end{cases}
\end{align*}
\begin{definition}\label{def:respectspartititon}
Let $M$ be a partition of $\{ (k, j): ~1 \leq j, k \leq N\}$.
\begin{enumerate}[(i)]
    \item Let $X$ be a vector space consisting of matrices $F$ for which the entries $F_{k, j}$ are functions $F_{k, j}: \Xi \to \mathbb C$. We say that $X$ \emph{respects the partition $M$} if the following two properties are satisfied:
    \begin{itemize}
        \item For every $m \in M$ and $(k, j), (k', j') \in m$: $F_{k, j} = F_{k', j'}$.
        \item For every $F \in X$ and $m \in M$ we have $F_m \in X$.
    \end{itemize} 
    \item Let $Y$ be a vector space consisting of linear operators on $\mathcal H$. We say that $Y$ \emph{respects the partition $M$} if the following two properties are satisfied:
    \begin{itemize}
        \item For every $m \in M$ and $(k, j), (k', j') \in m$: $A_{k,j} = \mathfrak A_{k, k'} A_{k', j'} \mathfrak A_{j', j}$.
        \item For every $A \in Y$ and $m \in M$ it is $A_m \in Y$. 
    \end{itemize}
\end{enumerate}
\end{definition}

The notion of regularity that we will need for the correspondence theory affiliated to a partition $M$ is the following:
\begin{definition}
Let $M$ be a partition of $\{ (j, k): ~1\leq j, k \leq N\}$. We say that $A \in \mathcal T^1(\mathcal H)$ is $M$-regular if every matrix entry $A_{j,k}$ of $A$ is regular and if they satisfy for every $m \in M$ that $\mathfrak A_{k',k} A_{k,j} \mathfrak A_{j, j'} = A_{k', j'}$ whenever $(k,j), (k', j') \in m$.
\end{definition}

In particular, $A \in \mathcal T^1(\mathcal H)$ is $M_{max}$-regular if and only if every matrix entry $A_{k,j}$ of $A$ is regular. $M_{max}$-regularity is therefore the weakest of these notions.

Regarding the following lemma, we note that the closed, $\alpha$-invariant subspace of $\mathcal T^1(\mathcal H)$ respecting the partition $M$, which is generated by some $A \in \mathcal T^1(\mathcal H)$, is identical with the closed, $\alpha$-invariant subspace generated by the collection $\{ A_m: ~m \in M\}$.
\begin{lemma}\label{lemma:l1reg}
Let $M$ be a partition of $\{ (k,j): ~1\leq j, k \leq N\}$. Further, let $A \in \mathcal T^1(\mathcal H)$ be $M$-regular. Then the smallest closed, $\alpha$-invariant subspace of $\mathcal T^1(\mathcal H)$ respecting the partition $M$ and containing $A$ is equal to
\begin{align*}
   \mathcal T^1(\mathcal H)_M := \{ B \in \mathcal T^1(\mathcal H): \forall m \in M ~\forall (k, j), (k', j') \in m: B_{k', j'} = \mathfrak A_{k', k} B_{k,j} \mathfrak A_{j, j'}\}.
\end{align*}
In particular, this space is the same for every $M$-regular operator and hence contains every $M$-regular operator.
\end{lemma}
\begin{proof}
It is clear by definition that $\mathcal T^1(\mathcal H)_M$ is a closed, $\alpha$-invariant subspace respecting the partition $M$ and containing $A$.

Conversely, let $B \in \mathcal T^1(\mathcal H)_M$ and $m \in M$. It suffices to prove that $B_m$ can be approximated by linear combinations of shifts of $A_m$. Let $(j_0, k_0) \in m$. By assumption, $A_{j_0, k_0}$ is regular and hence Theorem \ref{thm:Correspondence_2} shows that linear combinations of shifts of $A_{j_0, k_0}$ can approximate $B_{j_0, k_0}$ in trace norm:
\begin{align*}
    B_{j_0, k_0} \approx \sum_\nu c_\nu \alpha_{z_\nu}(A_{j_0, k_0}).
\end{align*}
But then, by the special structure of $A$ and $B$, we have
\begin{align*}
    B_{j,k} = \mathfrak A_{k_0, k} B_{j_0, k_0} \mathfrak A_{j, j_0}\approx  \sum_\nu c_\nu \alpha_{z_\nu}(A_{k_0, k} A_{j_0, k_0}A_{j, j_0}) = \sum_\nu c_\nu \alpha_{z_\nu}(A_{j, k})
\end{align*}
for $(j, k) \in m$. Therefore, linear combinations of shifts of $A_m$ approximate $B_m$.
\end{proof}

With the previous lemma at hand, we can now show the desired correspondence theorem:
\begin{theorem}\label{corr:thm:final}
Let $M$ be a partition of $\{ (j, k): ~1\leq j, k \leq N\}$. 
\begin{enumerate}[(1)]
\item There is a one to one correspondence between closed, $\alpha$-invariant subspaces $\mathcal D_0$ of $\operatorname{BUC}(\Xi)^{N \times N}$ which respect $M$ and $\alpha$-invariant closed subspaces $\mathcal D_1$ of $\mathcal C_1(\mathcal H)$ which respect $M$. The correspondence is given by
\begin{align*}
    \mathcal D_1 = \overline{\mathcal T^1(\mathcal H)_M \ast \mathcal D_0}, \quad \mathcal D_0 = \overline{\mathcal T^1(\mathcal H)_M \ast \mathcal D_1}.
\end{align*}
\item If $A \in \mathcal T^1(\mathcal H)$ is $M$-regular, then $\mathcal D_0 = \overline{A \ast \mathcal D_1}$ and $\mathcal D_1 = \overline{A \ast \mathcal D_0}$. By (1), this is independent of the choice of $A$.
\item Let $\mathcal D_0, \mathcal D_1$ be corresponding spaces in the above sense. 
\begin{enumerate}[(a)]
    \item Given $F \in \operatorname{BUC}(\Xi)^{N \times N}$, we have $F \in \mathcal D_0$ if and only if $A \ast F \in \mathcal D_1$.
    \item Given $B \in \mathcal C_1(\mathcal H)$, we have $B \in \mathcal D_1$ if and only if $A \ast B \in \mathcal D_0$.
\end{enumerate}
\end{enumerate}
\end{theorem}

\begin{proof}
Having discussed Theorems \ref{cor:maxpart} and \ref{cor:minpart} before, the proof is left to the reader as it does not need any new ideas.
\end{proof}

For our purposes, the most important outcome of the previous result is the following:

\begin{corollary} \label{cor:thm:final}
Let $A \in \mathcal T^1(\mathcal H)$ be $M_{max}$-regular and $B \in \mathcal L(\mathcal H)$. Then, $B$ is compact if and only if $B \in \mathcal C_1(\mathcal H)$ and $A \ast B \in C_0(\Xi)^{N \times N}$.
\end{corollary}

\begin{proof}
By Corollary \ref{lem:correspondence_compact} we know that for $\mathcal D_0^{j, k} = C_0(\Xi)$ we obtain $\mathcal D_1^{j,k} = \mathcal K(\mathcal H_j, \mathcal H_k)$. Therefore, with respect to the partition $M_{max}$ we have $\mathcal D_1 = \mathcal K(\mathcal H)$ and $\mathcal D_0 = C_0(\Xi)^{N \times N}$. The statement now follows from the previous theorem.
\end{proof}
\begin{remark} By now, one could get the impression that there is a finite number of correspondence theorems, one for each partition. But there are even more forms of regularity which yield similar results. Just to give one example, for $N = 2$ we could say that $A$ is regular if every matrix entry is regular and $A_{2,1} = 2 \mathfrak A_{2,1} A_{1,1} $, $A_{1,2} =  A_{1,1} \mathfrak A_{1,2}$ and $A_{2,2} = \mathfrak A_{1,2} A_{1,1} \mathfrak A_{2,1}$. This leads to a similar notion of subspaces of $\mathcal T^1(\mathcal H)$ respecting the partition $M = M_{min}$, that is, one has to consider those operators $B \in \mathcal T^1(\mathcal H)$ the entries of which satisfy the same algebraic relations: $B_{2,1} = 2 \mathfrak A_{2,1} B_{1,1} $, $B_{1,2} =  B_{1,1} \mathfrak A_{1,2}$ and $B_{2,2} = \mathfrak A_{1,2} B_{1,1} \mathfrak A_{2,1}$. Analogously, this algebraic relation carries over to the correspondence theory. Of course, one could consider many other algebraic relations between the matrix entries of $A$, similar to the above example, and each would give a different correspondence theory. 
\end{remark}

To make things even more complicated, there are other variants of the correspondence theory besides the just-mentioned variants of Theorem  \ref{corr:thm:final}. We describe one more such variant. Let $M$ be a fixed partition. We denote by $L^\infty(\Xi)^M$ the subspace of $L^\infty(\Xi)^{N \times N}$ consisting of those $F = (F_{k,j})_{j,k=1,\ldots,N}$ respecting the partition $M$. We denote by $m_{reg}$ a subset of $\{ (j, k): ~1 \leq j, k \leq N\}$ containing one tuple $(j,k)$ from each $m \in M$. For $(j_0, k_0) \in m_{reg}$ we denote by $m_{j_0,k_0}$ the element from $M$ which contains $(j_0, k_0)$. Further, we let $B \in \mathcal T^1(\mathcal H)$ be such that $B_{j_0,k_0}$ is regular for every $(j_0, k_0) \in m_{reg}$. Set
\begin{align*}
    \mathcal A_{B, M} = \overline{B \ast L^\infty(\Xi)^M}.
\end{align*}
This is clearly a closed, $\alpha$-invariant subspace of $\mathcal C_1(\mathcal H)$ ($\alpha$-invariance follows from Eq.\ \eqref{eq:shiftconv2}). But note that it does not respect the partition $M$ in the sense of Definition \ref{def:respectspartititon}, as it can happen that $\mathfrak A_{k',k} B_{k,j} \mathfrak A_{j, j'} \neq B_{k', j'}$ for $(k,j), (k', j') \in m$.

Further, for each $m \in M$ let $\mathcal D_0^m$ be a closed, $\alpha$-invariant subspace of $\operatorname{BUC}(\Xi)$ and set
\begin{align*}
    \mathcal D_0 := \{ F \in \operatorname{BUC}(\Xi)^{N \times N}: ~F_{k,j} \in \mathcal D_0^m \text{ for every } (j,k) \in m, m \in M\}.
\end{align*}
This is also an $\alpha$-invariant closed subspace of $\operatorname{BUC}(\Xi)^{N \times N}$. Clearly, it respects the partition $M_{max}$. As before, we denote by $\mathcal D_1$ the space corresponding to $\mathcal D_0$ in the sense of Theorem \ref{corr:thm:final} with respect to the partition $M_{max}$. This yields the following important result:

\begin{theorem} \label{thm:refined_correspondence_theorem}
Let the notation be as in the preceding paragraph. Further, let $A \in \mathcal T^1(\mathcal H)$ be such that $A_{k_0, j_0}$ is regular for $(j_0, k_0) \in m_{reg}$. Then, for every $C \in \mathcal A_{B, M}$ the following statements are equivalent:
\begin{enumerate}[(1)]
    \item $C \in \mathcal D_1$.
    \item $A_{k_0, j_0} \ast C_{k_0, j_0}\in \mathcal D_0^{m_{j_0,k_0}}$ for every $(j_0, k_0) \in m_{reg}$.
\end{enumerate}
\end{theorem}

\begin{proof}
By density, it suffices to work with operators of the form $C = B \ast F$ with $F \in L^\infty(\Xi)^M$. 

(1) $\Rightarrow$ (2): By Theorem \ref{corr:thm:final}, if $B \ast F \in \mathcal D_1$, then  $A \ast (B \ast F) \in \mathcal D_0$, that is, for each $(j,k) \in m$, $m \in M$ we have $A_{k,j} \ast (B_{k,j} \ast F_{k,j}) \in \mathcal D_0^m$. This clearly implies (2).

(2) $\Rightarrow$ (1): Let $(j, k) \in m_{j_0, k_0}$ and assume that $A_{k_0, j_0} \ast B_{k_0, j_0} \ast F_{k_0, j_0} \in \mathcal D_0^{m_{j_0, k_0}}$. Then, by regularity of $A_{k_0, j_0}$ and Theorem \ref{thm:corr_1}, $B_{k_0, j_0} \ast F_{k_0, j_0} \in \mathcal D_1^{m_{j_0, k_0}}$. Since $B_{k_0, j_0}$ is regular, we can approximate $\mathfrak A_{k_0, k} B_{k, j} \mathfrak A_{j, j_j}$ in trace norm by finite sums of the form
\begin{align*}
    \mathfrak A_{k_0, k} B_{k, j} \mathfrak A_{j, j_0} \approx \sum\limits_{\nu} c_\nu \alpha_{z_\nu}(B_{k_0, j_0})
\end{align*}
with $c_{\nu} \in \C$, $z_{\nu} \in \Xi$. Hence, by properties of the convolution (see Eq.~\eqref{eq:convolution_estimate_L^inf} and \eqref{eq:shiftconv1}) and since $F_{k,j} = F_{k_0,j_0}$ by assumption, we can approximate $B_{k, j} \ast F_{k,j}$ in operator norm:
\begin{align*}
    B_{k, j} \ast F_{k,j} &\approx \sum\limits_\nu \mathfrak A_{k, k_0} \alpha_{z_\nu}(B_{k_0, j_0}) \mathfrak A_{j_0, j} \ast F_{k,j}\\
    &= \sum\limits_\nu \mathfrak A_{k, k_0} \alpha_{z_\nu}(B_{k_0, j_0} \ast F_{k_0,j_0}) \mathfrak A_{j_0, j} \\
    &\in \mathfrak A_{k, k_0} \mathcal D_1^{m_{j_0, k_0}} \mathfrak A_{j_0, j}.
\end{align*}
By Corollary \ref{cor:equalspaces}, we have $ \mathfrak A_{k, k_0} \mathcal D_1^{m_{j_0, k_0}} \mathfrak A_{j_0, j} = \mathcal D_1^{j, k}$. Hence, we see that $B_{k, j} \ast F_{k,j} \in \mathcal D_1^{j, k}$. Since this holds for any pair $(j, k)$, we have proven $C = B \ast F \in \mathcal D_1$.
\end{proof}

We want to stress that for $N = 1$ we have $\mathcal A_{B, M} = \mathcal C_1$, so all the correspondence theorems we have discussed collapse to the same result in this case.  Moreover, if we choose $\mathcal D_0^m := C_0(\Xi)$ for every $m \in M$ in the construction of the previous theorem (that is, $\mathcal D_0 = C_0(\Xi)^{N \times N}$), we obtain $\mathcal D_1 = \Kc(\Hc)$ like in Corollary \ref{cor:thm:final} and therefore:

\begin{corollary} \label{cor:compact}
Let $C \in \mathcal A_{B,M}$ and $A \in \Tc^1(\Hc)$ such that $A_{k_0, j_0}$ is regular for $(k_0, j_0) \in m_{reg}$. Then the following statements are equivalent:
\begin{enumerate}[(1)]
    \item $C$ is compact.
    \item $C_{j_0, k_0}$ is compact for every $(j_0, k_0) \in m_{reg}$.
    \item $A_{k_0, j_0} \ast C_{j_0, k_0} \in C_0(\Xi)$ for every $(j_0, k_0) \in m_{reg}$.
\end{enumerate}
\end{corollary}

\section{Polyanalytic Fock spaces}\label{sec:3}

We first recall some notation and basic results from \cite{hagger2023,Vasilevski2000} and then apply the results from the previous section to this case. Let $\mu$ be the Gaussian measure given by $\mathrm{d}\mu(z) = \frac{1}{\pi}e^{-|z|^2} \, \mathrm{d}z$ on $\mathbb C \cong \mathbb R^2$. We say that a smooth function $f: \mathbb C \to \mathbb C$ is polyanalytic of order at most $n \in \mathbb N$ if: 
\begin{align*}
    \frac{\partial^n f}{\partial \overline{z}^n} = 0.
\end{align*}
The polyanalytic Fock space of order $n$ is now defined as the closed subspace \label{def:F_n}\nomenclature{$F_n^2$}{Polyanalytic Fock space of order $n$, p.~\pageref{def:F_n}}$F^2_n$ of $L^2(\mathbb C, \mu)$ consisting of polyanalytic functions of order at most $n$. Further, we define the spaces $F^2_{(1)} := F^2_1$ and \label{def:F_k}\nomenclature{$F_{(k)}^2$}{True polyanalytic Fock space of order $k$, p.~\pageref{def:F_k}}$F^2_{(k)} := F^2_k \ominus F^2_{k-1}$ for $k = 2,\ldots,n$, which are called \emph{true polyanalytic Fock spaces}. $F^2_{(1)}$ is of course just the standard Fock space of analytic functions. By definition, $F_n^2$ can be written as an orthogonal sum of true polyanalytic Fock spaces:
\begin{align*}
    F_n^2 = \bigoplus_{k=1}^n F_{(k)}^2.
\end{align*}
The orthogonal projection onto $F^2_{(k)}$ will be denoted by $P_{(k)}$. Moreover, $P_n := \sum\limits_{k = 1}^n P_{(k)}$ is the orthogonal projection onto $F^2_n$. Each $F_{(k)}^2$ and $F_n^2$ is a reproducing kernel Hilbert space and the reproducing kernels are given by
\begin{align}\label{def:repr_kernel1}
    \nomenclature{$K_{w,(k)}$}{Reproducing kernels of $F_{(k)}^2$, Eq.~\eqref{def:repr_kernel1} on p.~\pageref{def:repr_kernel1}}K_{w,(k)}(z) := K_{(k)}(z, w) := L_{k-1}^0 (|z-w|^2)e^{z \overline{w}}
\end{align}
and
\begin{align}\label{def:repr_kernel2}
\nomenclature{$K_{w,n}$}{Reproducing kernels of $F_{n}^2$, Eq.~\eqref{def:repr_kernel2} on p.~\pageref{def:repr_kernel2}}K_{w,n}(z) := K_n(z, w) := L_{n-1}^1 (|z-w|^2)e^{z \overline{w}},
\end{align}
respectively. Here, for $k,\alpha \in \N_0$, $L_k^\alpha$ denotes the generalized Laguerre polynomial, which is defined by
\begin{align}\label{def:Laguerre}
    \nomenclature{$L_k^\alpha$}{Generalized Laguerre polynomials, Eq.~\eqref{def:Laguerre} on p.~\pageref{def:Laguerre}}L_k^\alpha(x) := \sum_{j=0}^k (-1)^j \binom{k+\alpha}{k-j}\frac{x^j}{j!}.
\end{align}
The reproducing kernels satisfy $\| K_{w,(k)}\| = e^{\frac{|w|^2}{2}}$. The normalized reproducing kernels $k_{w,(k)}$ and $k_{w,n}$ are therefore given by 
\begin{align}\label{def:normalized_kernel1}
\nomenclature{$k_{w,(k)}$}{Normalized reproducing kernels of $F_{(k)}^2$, Eq.~\eqref{def:normalized_kernel1} on p.~\pageref{def:normalized_kernel1}}k_{w,(k)}(z) := \frac{K_{w,(k)}(z)}{\| K_{w,(k)}\|} = L_{k-1}^0 (|z-w|^2)e^{z \overline{w}-\frac{|w|^2}{2}}
\end{align}
and
\begin{align}\label{def:normalized_kernel2}
\nomenclature{$k_{w,n}$}{Normalized reproducing kernels of $F_{n}^2$, Eq.~\eqref{def:normalized_kernel2} on p.~\pageref{def:normalized_kernel2}}k_{w,n}(z) := \frac{K_{w,n}(z)}{\| K_{w,n}\|} = \frac{1}{\sqrt{n}}L_{n-1}^1 (|z-w|^2)e^{z \overline{w}-\frac{|w|^2}{2}},
\end{align}
respectively. Given $f \in L^\infty(\mathbb C)$, the Toeplitz operator $T_{f, (k)}$ on $F_{(k)}^2$ is defined as \label{def:Toeplitz}\nomenclature{$T_{f, (k)}$}{Toeplitz operator on $F_{(k)}^2$, p.~\pageref{def:Toeplitz}}$T_{f, (k)}(g) = P_{(k)}(fg)$. Similarly, the Toeplitz operator \label{def:Toeplitz2}\nomenclature{$T_{f, n}$}{Toeplitz operator on $F_{n}^2$, p.~\pageref{def:Toeplitz2}}$T_{f, n}^2$ on $F_n^2$ is defined by $T_{f, n}(g) = P_n(fg)$. The Weyl operators $W_z \from L^2(\mathbb C, \mu) \to L^2(\mathbb C, \mu)$ are defined as follows:
\begin{equation} \label{eq:Weyl_operators_concrete}
    W_z f(w) = e^{w\overline{z} - \frac{|z|^2}{2}} f(w-z).
\end{equation}
They are unitary and satisfy
\begin{align*}
    W_z^\ast = W_{-z}, \quad W_z W_w = e^{-i\sigma(z, w)/2}W_{z+w}, \quad z, w \in \mathbb C.
\end{align*}
Here, $\sigma$ is the symplectic form on $\mathbb C \cong \mathbb R^2$ given by $\sigma(z, w) =  2\operatorname{Im}(z\overline{w})$. Each $W_z$ leaves $F_{(k)}^2$ invariant for every $k \in \mathbb N$ (cf. \cite[Proposition 8]{hagger2023}). As it is the same for each $k$, we will just write $W_z$ instead of $W_z^k$ for the Weyl operators acting on $F_{(k)}^2$. Using
\begin{align*}
    (W_zk_{w,(k)})(v) &= e^{v \overline{z} - \frac{|z|^2}{2}} L_{k-1}^0 (|(v - z) - w|^2) e^{(v-z)\overline{w} - \frac{|w|^2}{2}}\\
    &= L_{k-1}^0(|v-(z+w)|^2)e^{v(\overline z+ \overline w)} e^{-\frac{|w|^2}{2} - \frac{|z|^2}{2} - \operatorname{Re}(z\overline{w}) - i\Im(z\overline w)}\\
    &= L_{k-1}^0(|v-(z+w)|^2)e^{v(\overline z+ \overline w)} e^{-\frac{|z+w|^2}{2} - i\Im(z\overline{w})},
\end{align*}
we see that
\begin{equation} \label{eq:shifted_kernels}
W_zk_{w,(k)} = e^{-i\Im(z\overline{w})} k_{z+w,(k)}.
\end{equation}
In particular, the span of the reproducing kernels $K_{w,(k)}$ is invariant under the action of the Weyl operators. We note the following fact, which is of course crucial for applying the methods discussed in Section \ref{sec:qha}.
\begin{proposition}
$(F_{(k)}^2, W)$ is an irreducible $\sigma$-representation of $(\mathbb R^2, \sigma)$.
\end{proposition}
\begin{proof}
    The result is well-known and we only give a very brief sketch the proof: By conjugating with the true $k$-polyanalytic Bargmann transform \cite{Abreu2010b}, the representation $(F_{(k)}^2, W)$ is unitarily equivalent to the representation $U_{(x, \xi)} f(t) = e^{2it\xi - ix\xi}f(t-x)$ on $L^2(\mathbb R)$, which is well-known to be irreducible.
\end{proof}
This result enables us to use the tools of QHA developed in the previous section to study operators on $F_{(k)}^2$. The parity operator $U$ is simply implemented by $Uf(z) = f(-z)$ here. Note that for $f \in F_{(k)}^2$ the following holds:
\begin{align*}
    \langle Uk_{z,(k)}, f\rangle = e^{-\frac{|z|^2}{2}}\langle K_{z,(k)}, f(-\cdot)\rangle = e^{-\frac{|z|^2}{2}} f(-z) = e^{-\frac{|z|^2}{2}} \langle K_{-z,(k)}, f\rangle = \langle k_{-z,(k)}, f\rangle,
\end{align*}
which yields
\begin{align}\label{eq:Uk}
    Uk_{z,(k)} = k_{-z,(k)}
\end{align}
for every $z \in \mathbb C$, $k = 1,\ldots,n$. Moreover, the intertwining operators $\mathfrak A_{j,k}$ between the different $F^2_{(k)}$ are implemented by
\begin{align*}
    \mathfrak A_{k+1,k} = \frac{1}{\sqrt{k}} \mathfrak{a}^\dag, \quad \mathfrak A_{k,k+1} = \frac{1}{\sqrt{k-1}} \mathfrak{a},
\end{align*}
where
\begin{align*}
    \mathfrak a^\dag = \left( -\frac{\partial}{\partial z} + \overline{z} \right), \quad \mathfrak a = \frac{\partial}{\partial \overline{z}}.    
\end{align*}
$\mathfrak A_{j,k}$ is then of course given by $\mathfrak A_{j,k} = \mathfrak  A_{j,j+1} \dots \mathfrak A_{k-1, k} $ or $\mathfrak A_{j,k} = \mathfrak A_{j, j-1} \dots \mathfrak A_{k+1, k}$, depending on whether $j < k$ or $k < j$.

Since we will make frequent use of the pairing of two normalized reproducing kernels, we write out the following readily verified formula: For $z, w \in \mathbb C$ we have
\begin{align}\label{eq:pairing_normkernel}
    \langle k_{z,(k)}, k_{w,(k)}\rangle = e^{-\frac{|z|^2 + |w|^2}{2}} L_{k-1}^0(|w-z|^2) e^{w\overline{z}}. 
\end{align}

In the following, for $f,g \in \Hc$, we will use tensor product notation $f \otimes g$ for the rank one operator on a Hilbert space $\mathcal H$ defined by $\varphi \mapsto (f \otimes g)(\varphi) := \langle \varphi, g\rangle f$. In particular, the tensor product is antilinear in the second entry. As a first result, we will derive an important identity related to the symplectic Fourier transform involving the Laguerre polynomials. For this, we recall that the Fourier--Weyl transform of the operator $A \in \mathcal T^1(F_{(k)}^2)$ is defined as the function $\mathcal F_W(A) \from \C \to \C$, $\mathcal F_W(A)(\xi) = \tr(AW_\xi)$. Among the most important properties of $\mathcal F_W$ are that it maps $\mathcal T^1(F_{(k)}^2)$ injectively and continuously to $C_0(\mathbb C)$ and extends to a unitary map from $\mathcal T^2(F_{(k)}^2)$ to $L^2(\mathbb C)$. Furthermore, it can be extended to an injective, weak$^\ast$ continuous map from $\mathcal L(F_{(k)}^2)$ to $\mathcal S'(\mathbb C)$, the tempered distributions on $\mathbb C \cong \mathbb R^2$. Details on the properties of $\mathcal F_W$ can be found in \cite{keyl_kiukas_werner16,werner84}.

\begin{proposition}\label{prop:convolution_identities}
Let  $z, w \in \mathbb C$.
\begin{enumerate}[(a)]
    \item The Fourier--Weyl transform of $k_{z,(k)} \otimes k_{w,(k)}$ is given by
    \begin{align*}
        \mathcal F_W(k_{z,(k)} \otimes k_{w,(k)})(\xi) = L_{k-1}^0 (|w-\xi-z|^2) e^{-\frac{|w - z - \xi|^2}{2} + i\Im(z\overline{\xi}) + i\Im(w( \overline{\xi} + \overline{z}))}.
    \end{align*}
    \item The convolution $(k_{z,(k)} \otimes k_{w,(k)}) \ast (k_{z,(k)} \otimes k_{w,(k)})$ is given by
    \begin{align*}
        (k_{z,(k)} &\otimes k_{w,(k)}) \ast (k_{z,(k)} \otimes k_{w,(k)})(u) \\
        &\quad = L_{k-1}^0(|w+z-u|^2)^2 e^{-|w+z-u|^2} e^{-2i(\Im(z\overline{u}) + \Im(u\overline{w}) + \Im(w \overline{z}))}.
    \end{align*}
    \item The following formula for the symplectic Fourier transform holds true:
\begin{align*}
    \mathcal F_\sigma&(L_{k-1}^0(|w+z-\cdot|^2)^2 e^{-|w+z-\cdot|^2} e^{-i(\sigma(z, \cdot) + \sigma(\cdot, w) + \sigma(w, z))})(\xi) \\
    &\quad = L_{k-1}^0 (|w-z-\xi|^2)^2 e^{-|w - z - \xi|^2 + 2i\Im(z\overline{\xi} +w \overline{\xi} + w \overline{ z})}.
\end{align*}
\end{enumerate}
\end{proposition}
\begin{proof}
We begin by computing the Fourier--Weyl transform of $k_{z,(k)} \otimes k_{w,(k)}$:
\begin{align*}
    \mathcal F_W(k_{z,(k)} \otimes k_{w,(k)})(\xi) &= \langle k_{z,(k)}, W_{-\xi} k_{w,(k)}\rangle\\
    &= \langle W_\xi k_{z,(k)}, k_{w,(k)}\rangle\\
    &= e^{-i\Im(\xi\overline{z})} \langle k_{\xi + z,(k)}, k_{w,(k)}\rangle
    \intertext{by Eq.\ \eqref{eq:shifted_kernels}. Applying Eq.\ \eqref{eq:pairing_normkernel} gives:}
    &= e^{-i\Im(\xi\overline{z})} e^{-\frac{|w|^2}{2} - \frac{|\xi + z|^2}{2}} L_{k-1}^0(|w - \xi - z|^2) e^{w(\overline \xi + \overline z)}\\
    &= L_{k-1}^0 (|w-\xi-z|^2) e^{-\frac{|w - z - \xi|^2}{2} + i\Im(z\overline{\xi}) + i\Im(w(\overline{\xi} + \overline{z}))}.
\end{align*}
Further, also using Eq.\ \eqref{eq:Uk}, we have:
\begin{align*}
    (k_{z,(k)} \otimes k_{w,(k)}) \ast (k_{z,(k)} \otimes k_{w,(k)})(u) &= \tr \left((k_{z,(k)} \otimes k_{w,(k)})W_u U (k_{z,(k)} \otimes k_{w,(k)}) U W_{-u}\right)\\
    &= \tr \left((k_{z,(k)} \otimes k_{w,(k)})(W_u k_{-z(k)} \otimes W_u k_{-w,(k)})\right)\\
    &= \langle W_u k_{-z,(k)}, k_{w,(k)}\rangle \langle k_{z,(k)}, W_u k_{-w,(k)}\rangle\\
    &= e^{i\Im(u(\overline{z}-\overline{w}))} \langle k_{u-z,(k)}, k_{w,(k)}\rangle \langle k_{z,(k)}, k_{u-w,(k)}\rangle\\
    &= e^{i\Im(u(\overline{z}-\overline{w}))} e^{-\frac{|u-z|^2}{2}-\frac{|w|^2}{2} - \frac{|z|^2}{2} - \frac{|u-w|^2}{2}}\\
    &\quad \times L_{k-1}^0(|w - u + z|^2)^2 e^{w(\overline u - \overline z) + (u-w)\overline z}\\
    &= L_{k-1}^0(|w+z-u|^2)^2 e^{-|w+z-u|^2} e^{-2i\Im(z\overline{u} + u\overline{w} + w\overline{z})}.
\end{align*}
By the convolution theorem \cite[Prop.\ 3.4(1)]{werner84}, 
\begin{align*}
    \mathcal F_\sigma((k_{z,(k)} \otimes k_{w,(k)}) \ast (k_{z,(k)} \otimes k_{w,(k)}))(\xi) = \mathcal F_W((k_{z,(k)} \otimes k_{w,(k)}))(\xi)^2,
\end{align*}
we obtain the formula in (c).
\end{proof}

\begin{corollary} \label{cor:regular}
The operator $k_{0,(k)} \otimes k_{0,(k)}$ is regular if and only if $k = 1$. For $k > 1$, $k_{0,(k)} \otimes k_{0,(k)}$ is $\infty$-regular.
\end{corollary}
\begin{proof}
By the previous proposition, the Fourier--Weyl transform of this operator is
\begin{align} \label{eq:k_0xk_0}
    \mathcal F_W(k_{0,(k)} \otimes k_{0,(k)})(\xi) = L_{k-1}^0(|\xi|^2) e^{-\frac{|\xi|^2}{2}}.
\end{align}
It therefore holds $\mathcal F_W(k_{0,(k)} \otimes k_{0,(k)})(\xi) = 0$ if and only if $L_{k-1}^0(|\xi|^2) = 0$. The Laguerre polynomial $L_{k-1}^0$ is free of zeroes if and only if $k = 1$. In all other cases, $L_{k-1}^0$ has a finite number of non-negative zeros, hence the zero set of  $L_{k-1}^0(|\cdot|^2)$ is a finite collection of disjoint concentric circles in the plane. This set of course has dense complement, so that $k_{0,(k)} \otimes k_{0,(k)}$ is $\infty$-regular for $k > 1$.
\end{proof}

\begin{theorem}\label{theorem:toeplitz_quant}
The map 
\[\Phi \from L^\infty(\mathbb C) \to \Lc(F^2_{(k)}), \quad f \mapsto (k_{0,(k)} \otimes k_{0,(k)}) \ast f\]
is injective if and only if $k = 1$. For $k > 1$, the kernel is exactly the weak$^\ast$ closure (with respect to the predual $L^1(\mathbb C)$) of the linear span of all characters $e^{i\sigma(\xi, \cdot)}$ with \nomenclature{$\Sigma_k$}{Zero set of $L_{k-1}^0(\abs{\xi}^2)$, p.~\pageref{theorem:toeplitz_quant}}$\xi \in \Sigma_k$, the zero set of $\xi \mapsto L_{k-1}^0(|\xi|^2)$.
\end{theorem}

\begin{proof}
Since $k_{0,(k)} \otimes k_{0,(k)}$ is regular for $k = 1$, it is well-known that $\Phi$ is injective for this case (cf.\ \cite{Fulsche2020,werner84}). For $k > 1$, we note that $e^{i\sigma(\xi, \cdot)}$ is in the kernel of $\Phi$ if and only if $\xi$ is in the zero set of $\mathcal F_W(k_{0,(k)} \otimes k_{0,(k)})$. 
Indeed, as justified by \cite[Proposition 5.9]{keyl_kiukas_werner16}, one can compute the Fourier--Weyl transform of the tempered operator $(k_{0,(k)} \otimes k_{0,(k)}) \ast f$, which yields a tempered distribution, as follows:
\begin{align*}
    \mathcal F_W((k_{0,(k)} \otimes k_{0,(k)}) \ast f) = \mathcal F_W(k_{0,(k)} \otimes k_{0,(k)}) \cdot \mathcal F_\sigma(f), 
\end{align*}
which has to be understood as the product of a Schwartz function and a tempered distribution. Now, for $f = e^{i\sigma(\xi, \cdot)}$ one readily verifies that $\mathcal F_\sigma(e^{i\sigma(\xi, \cdot)}) = \pi\delta_\xi$, where $\delta_\xi$ is the delta distribution supported at $\xi$. Hence,
\begin{align}\label{eq:fourierweyl_exponential}
    \mathcal F_W((k_{0,(k)} \otimes k_{0,(k)}) \ast e^{i\sigma(\xi, \cdot)}) = \pi \mathcal F_W(k_{0,(k)} \otimes k_{0,(k)}) \cdot \delta_\xi,
\end{align}
which is the zero distribution if and only if $\xi$ is a zero of $\mathcal F_W(k_{0,(k)} \otimes k_{0,(k)}) =  L_{k-1}^0(|\cdot|^2) e^{-\frac{|\cdot|^2}{2}}$.

Further, note that the map $\Phi \from L^\infty(\mathbb C) \to \Lc(F^2_{(k)})$ is weak$^\ast$ continuous (with respect to the preduals $L^1(\mathbb C)$ and $\mathcal T^1(F_{(k)}^2)$) which intertwines the shifts: 
\begin{align*} 
\alpha_z((k_{0,(k)} \otimes k_{0,(k)})\ast f ) = (k_{0,(k)} \otimes k_{0,(k)}) \ast \alpha_z(f).
\end{align*}
Therefore, the kernel of $\Phi$ is an $\alpha$-invariant weak$^\ast$ closed subspace of $L^\infty(\mathbb C)$. For any such subspace $X \subseteq L^\infty(\mathbb C)$ define $\Sigma(X) := \{ \xi \in \mathbb C: ~e^{i\sigma(\xi, \cdot)} \in X\}$. As we have just seen, the characters $e^{i\sigma(\xi, \cdot)}$ in the kernel of $\Phi$ are exactly those with $L_{k-1}^0(|\xi|^2) = 0$. Hence, in the language of spectral synthesis (see e.g. \cite[Chapter 40]{Hewitt_Ross2} for an introduction to the problem of spectral synthesis),
\begin{align*}
    \Sigma_k = \Sigma ( \{ f \in L^\infty(\mathbb C): ~(k_{0,(k)} \otimes k_{0,(k)}) \ast f = 0\}) = \bigcup_{r \geq 0: L_{k-1}^0(r) = 0} \sqrt{r} S^1.
\end{align*}
The set $\Sigma_k$ is a finite union of concentric circles. \cite[Theorem 2.7.6]{Reiter_Stegeman2000} shows that circles are sets of spectral synthesis and it therefore follows from \cite[Theorem 8]{Muraleedharan_Parthasarathy} that $\Sigma_k$ is also a set of spectral synthesis, which just means that there is exactly one $\alpha$-invariant, weak$^\ast$ closed subset $X$ of $L^\infty(\mathbb C)$ with $\Sigma(X) = \Sigma_k$. Since $X_k := \overline{\textrm{span}} \{ e^{i\sigma(\xi, \cdot)}: ~\xi \in \Sigma_k\}$ also satisfies $\Sigma(X_k) = \Sigma_k$, we obtain that the kernel is exactly $X_k$.
\end{proof}
\begin{remark}
    In general, not much concrete can be said about the zero sets of the Laguerre polynomials $L_n^0$, which appear above. Among the most noteworthy results, an irreducibility theorem by Schur \cite{Schur1929} implies that for $n \geq 2$, all zeros of $L_n^0$ are irrational.
\end{remark}

Analogously, we have:
\begin{theorem}\label{thm:kernel_berezin_trafo}
The map 
\[\Psi \from \Lc(F^2_{(k)}) \to L^{\infty}(\C), \quad A \mapsto (k_{0,(k)} \otimes k_{0,(k)}) \ast A\]
is injective if and only if $k = 1$. For $k > 1$, the kernel is exactly the weak$^\ast$ closure (with respect to the predual $\mathcal T^1(F_{(k)}^2)$) of the linear span of all Weyl operators $W_\xi$ with $\xi \in \Sigma_k$, the zero set of $\xi \mapsto L_{k-1}^0(|\xi|^2)$.
\end{theorem}

\begin{proof}
The reasoning is completely analogous to the previous proof: As before, the kernel of the map is a weak$^\ast$ closed and $\alpha$-invariant subspace of $\mathcal L(F_{(k)}^2)$. By the convolution theorem,
\begin{align*}
    \mathcal F_\sigma((k_{0,(k)} \otimes k_{0,(k)}) \ast A) = \mathcal F_W(k_{0,(k)} \otimes k_{0,(k)})\cdot \mathcal F_W(A), 
\end{align*}
which is again to be interpreted in the sense of distributions. For the Fourier--Weyl transform of the Weyl operators and for an appropriate operator $A$ (say, $A = \mathcal F_W^{-1}(g)$ for some Schwartz function $g$ on $\mathbb C$) we have:
\begin{align*}
    \mathcal F_W(A)(-\xi) = \tr(AW_{-\xi}) = \langle A, W_\xi\rangle_{\mathcal T^2} = \langle \mathcal F_W(A), \mathcal F_W(W_{\xi})\rangle_{L^2},
\end{align*}
which shows that $\mathcal F_W(W_\xi) = \delta_{-\xi}$. Here, the pairings are the extensions of the Hilbert--Schmidt and $L^2$ inner products. Since the Fourier--Weyl transform maps $\mathcal T^2(F_{(k)}^2)$ unitarily to $L^2(\mathbb C)$ \cite[Lemma 2.7]{keyl_kiukas_werner16}, it clearly satisfies the above version of Plancherel's identity for the pairing of two Hilbert-Schmidt operators. The extension of Plancherel's identity to the pairing between trace class and bounded operators that we used follows easily from standard density arguments. Hence, we have $W_\xi \ast (k_{0,(k)} \otimes k_{0,(k)}) = 0$ if and only if $-\xi \in \Sigma_k$. Hence, the spectrum of the kernel in the sense of quantum spectral synthesis \cite{Fulsche_Rodriguez2023} is given by
\begin{align*}
    \Sigma(\{ A \in \mathcal L(F_{(k)}^2): ~(k_{0,(k)} \otimes k_{0,(k)}) \ast A = 0\}) = \Sigma_k,
\end{align*}
where $\Sigma(X) := \{ \xi \in \mathbb C: ~W_\xi \in X\}$ for any weak$^\ast$ closed and $\alpha$-invariant subspace $X$ of $\mathcal L(F_{(k)}^2)$. Since quantum spectral synthesis is equivalent to spectral synthesis (cf.\ \cite[Theorem 2.2]{Fulsche_Rodriguez2023}), the kernel is given by the weak$^\ast$ closure of the linear span of all $W_\xi$ with $\xi \in \Sigma_k$.
\end{proof}

The function $\Psi(A)$ is commonly known as the Berezin transform $\widetilde{A}$ of $A$. Indeed, by \eqref{eq:shifted_kernels} and \eqref{eq:Uk}, we have
\begin{align} \label{eq:Berezin}
    (k_{0,(k)} \otimes k_{0,(k)}) \ast A &= \tr((k_{0,(k)} \otimes k_{0,(k)}) W_z U_k A U_k W_{-z}) = \tr((k_{z,(k)} \otimes k_{z,(k)})A)\notag\\
    &= \sp{Ak_{z,(k)}}{k_{z,(k)}} = \widetilde{A}.
\end{align}
Note that this is different from the generalized Berezin transform defined in \cite[Definition 15]{hagger2023}, which is given by $(l_{0,k} \otimes l_{0,k}) \ast A$ with $l_{z,k} = \Af_{k,1}k_{z,(1)}$ for $z \in \C$. The generalized Berezin transform has the benefit that $l_{0,k} \otimes l_{0,k}$ is regular while $k_{0,(k)} \otimes k_{0,(k)}$ is not (see Corollary \ref{cor:regular}). In particular, $A \mapsto (l_{0,k} \otimes l_{0,k}) \ast A$ is injective. This can be used, for instance, to show that
\begin{equation} \label{eq:Toeplitz}
    (k_{0,(k)} \otimes k_{0,(k)}) \ast f = \pi T_{f,(k)}
\end{equation}
for $f \in L^{\infty}(\C)$. Indeed, convolving each side with $l_{0,k} \otimes l_{0,k}$ results in $z \mapsto \sp{fl_{z,k}}{l_{z,k}}$ on both sides (cf.\ \cite[Proposition 2.12]{Fulsche2020}).

\begin{remark} \label{rem:Question31}
The previous two theorems show that for $k > 1$, there is a non-zero symbol $f \in L^{\infty}(\C)$ such that $T_{f,(k)} = 0$ and a non-zero operator $A \in \Cc_1$ such that $\widetilde{A} = 0$, respectively. This also answers \cite[Question 31]{hagger2023}, which asked whether the compactness of $T_{f,(k)}$ is independent of $k$, somewhat spectacularly in the negative. While the compactness of $T_{f,(1)}$ always implies the compactness of $T_{f,(k)}$ (see \cite[Theorem 23]{hagger2023} and also the remarks after Theorem \ref{thm:compactness_n} below), if we set $f(z) := e^{i\sigma(z,\xi) + \frac{\abs{\xi}^2}{2}}$ with $L_{k-1}^0(\abs{\xi}^2) = 0$, then $T_{f,(k)} = 0$ and $T_{f,(1)} = W_{\xi}$.
\end{remark}

As $l_{0, j} \otimes l_{0, k} = \mathfrak A_{j,1} (k_{0,(1)} \otimes k_{0,(1)}) \mathfrak A_{1,k}$ is a regular operator, the Correspondence Theorem \ref{thm:corr_1} implies the following result, which is analogous to \cite[Theorem 16]{hagger2023}:

\begin{theorem} \label{thm:compactness}
Let $j, k \in \mathbb N$. Then, $A \in \mathcal L(F_{(j)}^2, F_{(k)}^2)$ is compact if and only if $A \in \mathcal C_1(F_{(j)}^2, F_{(k)}^2)$ and $\big((l_{0, j} \otimes l_{0, k}) \ast A\big)(z) = \langle A l_{z,j}, l_{z,k}\rangle \to 0$ as $\abs{z} \to \infty$.
\end{theorem}

In the next section we will show that the condition ``$AP_{(j)} \in \BDO^2$'' from \cite[Theorem 16]{hagger2023} is actually equivalent to $A \in \mathcal C_1^{k,j}$, which shows that these results are really the same.

For $F^2_n$ we have a similar result, which is in fact a bit stronger than \cite[Theorem 16]{hagger2023}. Indeed, Corollary \ref{cor:compact} implies:

\begin{theorem} \label{thm:compactness_n}
Let $M$ be a partition of $\set{(j,k) : 1 \leq j,k \leq n}$ and let $m_{reg}$ be a set that contains one element from each $m \in M$. Further assume that $A,B \in \Tc^1(F^2_n)$ are operators such that $P_{(j_0)}A|_{F^2_{(k_0)}},P_{(j_0)}B|_{F^2_{(k_0)}} \in \Tc^1(F^2_{(k_0)},F^2_{(j_0)})$ are regular for each $(j_0,k_0) \in m_{reg}$. Then for $C \in \overline{B \ast L^{\infty}(\C)^M}$ the following are equivalent:
\begin{enumerate}[(1)]
    \item $C$ is compact.
    \item $C_{j_0, k_0}$ is compact for every $(j_0, k_0) \in m_{reg}$.
    \item $A_{k_0, j_0} \ast C_{j_0, k_0} \in C_0(\C)$ for every $(j_0, k_0) \in m_{reg}$.
\end{enumerate}
\end{theorem}

If we set $A := B := \sum\limits_{j,k = 1}^n l_{0,j} \otimes l_{0,k}$, $M := M_{max}$, then $\overline{B \ast L^{\infty}(\C)^M}$ is equal to $\Cc_1(F^2_n)$. Therefore we obtain that $C \in \Lc(F^2_n)$ if and only if $C \in \Cc_1(F^2_n)$ and $\langle C l_{z, j}, l_{z, k}\rangle \to 0$ as $\abs{z} \to \infty$ for every $j,k = 1,\ldots,N$, which is exactly the result in \cite[Theorem 16]{hagger2023}. However, if we set $A := B := \sum\limits_{j,k = 1}^n k_{0,(j)} \otimes k_{0,(k)}$ and $M := M_{min}$, then $\overline{B \ast L^{\infty}(\C)^M} = \overline{\set{T_{f,n} : f \in L^{\infty}(\C)}}$. In this case Theorem \ref{thm:compactness_n} implies that $T_{f,n}$ is compact if and only if $T_{f,(1)}$ is compact, which is \cite[Theorem 23]{hagger2023}. This shows that Theorem \ref{thm:compactness_n} is indeed a generalization of the results in \cite{hagger2023}.

While we showed in Theorems \ref{theorem:toeplitz_quant} and \ref{thm:kernel_berezin_trafo} that neither the Toeplitz quantization nor the Berezin transform is injective for $k > 1$, we still want to emphasize the following, which is an immediate consequence of the $\infty$-regularity of $k_{0,(k)} \otimes k_{0,(k)}$ and Theorem \ref{thm:inftyregular}:

\begin{proposition}\label{prop:toeplitz_quant_wstar_dense}
Let $k \geq 1$.
\begin{enumerate}[(1)]
    \item The map $L^1(\mathbb C) \ni f \mapsto T_{f, (k)} \in \mathcal T^1(F_{(k)}^2)$ is injective.
    \item The map $\mathcal T^1(F_{(k)}^2) \ni A \mapsto \widetilde{A} \in L^1(\mathbb C)$ is injective.
    \item The map $L^\infty(\C) \ni f \mapsto T_{f, (k)} \in \mathcal L(F_{(k)}^2)$ has dense range (with respect to weak$^\ast$ topology).
    \item The map $\mathcal L(F_{(k)}^2) \ni A \mapsto \widetilde{A} \in L^\infty(\mathbb C)$ has dense range (with respect to weak$^\ast$ topology).
\end{enumerate}
\end{proposition}

\begin{remark}\label{rem:Weyl_as_Toeplitz}
Let us mention that the density of the range in weak$^\ast$ topology can also be obtained more constructively. Since the complement of the set of zeros of the function $L_{k-1}^0(|\cdot|)^2$ is dense in $\mathbb C$, one can argue as follows: Let $z \in \mathbb C$ with $L_{k-1}^0(|z|^2) \neq 0$. Then, by Eqs. \eqref{eq:k_0xk_0}, \eqref{eq:fourierweyl_exponential} and \eqref{eq:Toeplitz}, 
\begin{align*}
    \mathcal F_W(\pi T_{e^{i\sigma(z, \cdot)}, (k)}) &= \mathcal F_W((k_{0,(k)} \otimes k_{0,(k)}) \ast e^{i\sigma(z, \cdot)}) = \mathcal F_W(k_{0,(k)} \otimes k_{0,(k)}) \cdot \pi \delta_z \\
    &= \pi L_{k-1}^0(|z|^2)e^{-\frac{|z|^2}{2}} \cdot \delta_z.
\end{align*}
On the other hand, we have that $\mathcal F_W(W_z) = \delta_{-z}$, which we observed in the proof of Theorem \ref{thm:kernel_berezin_trafo}. We therefore arrive at
\begin{align*}
    \mathcal F_W(\pi T_{e^{i\sigma(z, \cdot)}, (k)}) = \pi L_{k-1}^0(|z|^2)e^{-\frac{|z|^2}{2}} \mathcal F_W(W_{-z}),
\end{align*}
which implies
\begin{align*}
    T_{g_z, (k)} = W_{z}
\end{align*}
for $g_z(w) := \frac{1}{L_{k-1}^0(|z|^2)}e^{\frac{|z|^2}{2}-i\sigma(z, w)}$
and $z \in \C$ with $L_{k-1}^0(|z|^2) \neq 0$. As $\set{z \in \C : L_{k-1}^0(|z|^2) \neq 0}$ is dense in $\mathbb C$ and $W_{z_m} \to W_z$ in weak$^\ast$ topology when $z_m \to z$, this shows that the weak$^\ast$ closure of the range of the Toeplitz quantization contains all Weyl operators. A weak$^\ast$ closed, $\alpha$-invariant subspace of $\mathcal L(F_{(k)}^2)$ containing all Weyl operators is known to be all of $\mathcal L(F_{(k)}^2)$, cf.\ \cite{Fulsche_Rodriguez2023}.
\end{remark}
Our next result is a extension of Proposition \ref{prop:toeplitz_quant_wstar_dense}. Even though we have seen in Theorem \ref{thm:kernel_berezin_trafo} that the Berezin transform is not injective on $\Lc(F^2_{(k)})$, it is in fact injective if restricted to Toeplitz operators. A similar statement also holds for the Toeplitz quantization in connection with Theorem \ref{theorem:toeplitz_quant}.

\begin{theorem} \label{thm:injectivity}
    The Berezin transform is injective on  $\set{ T_{f, (k)}: ~f \in L^\infty(\C)}$. The Toeplitz quantization is injective on $\{ \widetilde{A}: ~A \in \mathcal L(F_{(k)}^2)\}$. 
\end{theorem}

\begin{proof}
    We abbreviate $g := \mathcal F_W(k_{0, (k)} \otimes k_{0, (k)}) = L_{k-1}^0(|\cdot|^2) e^{-\frac{|\cdot|^2}{2}}$. By \eqref{eq:Berezin}, \eqref{eq:Toeplitz} and the convolution theorem, the kernel of $f \mapsto \widetilde{T_{f, (k)}}$ is the same as the kernel of $f \mapsto g^2 \mathcal F_\sigma(f)$. The same argument as in the proof of Theorem \ref{theorem:toeplitz_quant} shows that this kernel agrees with the weak$^\ast$ closure of $\spann\set{e^{i\sigma(\xi, \cdot)} : g^2(\xi) = 0}$. But $g^2(\xi) = 0$ if and only if $\xi \in \Sigma_k$. Hence, if $\widetilde{T_{f, (k)}} = 0$, then we already have $T_{f, (k)} = 0$. 

    Analogously, one can see that $T_{\widetilde{A}, (k)} = 0$ already implies $\widetilde{A} = 0$.
\end{proof}

In Remark \ref{rem:Weyl_as_Toeplitz} we have seen that almost all of the Weyl operators $W_\xi$ can still be written as Toeplitz operators with bounded symbols. Theorem \ref{thm:injectivity} shows that this is impossible for those $\xi \in \C$ with $L_{k-1}^0(|\xi|^2) = 0$. Also, there is no bounded operator that has one of the corresponding characters as its Berezin transform.

\begin{corollary} \label{cor:Weyl_not_Toeplitz}
    Let $\xi \in \Sigma_k$, that is, $L_{k-1}^0(|\xi|^2) = 0$. Then there exists no $f \in L^\infty(\C)$ such that $T_{f, (k)} = W_{\xi}$. Further, there exists no $A \in \mathcal L(F_{(k)}^2)$ such that $\widetilde{A} = e^{i\sigma(\xi, \cdot)}$.
\end{corollary}

\begin{proof}
    Assume that there was such an $f \in L^\infty(\C)$. Since $\widetilde{W_\xi} = 0$, we would have $\widetilde{T_{f, (k)}} = 0$, hence $T_{f, (k)} = 0 \neq W_\xi$ by the previous result. The other statement is obtained analogously.
\end{proof}

In a similar spirit we can show that, despite its downsides indicated in Remark \ref{rem:Question31}, the Berezin transform can still be used to characterize compact Toeplitz operators. This answers a variant of \cite[Question 32]{hagger2023} for the true polyanalytic Fock spaces $F^2_{(k)}$.

\begin{theorem}~ \label{thm:compact_Toeplitz_operators}
\begin{enumerate}[(1)]
    \item Let $f \in L^\infty(\C)$. Then, $T_{f, (k)} \in \mathcal K(F_{(k)}^2)$ if and only if $\widetilde{T_{f, (k)}} \in C_0(\C)$.
    \item Let $A \in \mathcal L(F_{(k)}^2)$. Then, $\widetilde{A} \in C_0(\C)$ if and only if $T_{\widetilde{A}, (k)} \in \mathcal K(F_{(k)}^2)$. 
\end{enumerate}
\end{theorem}

\begin{proof}
    We only prove (1) as the second statement follows analogously.
    
    Clearly, compactness of $T_{f, (k)}$ implies that $\widetilde{T_{f, (k)}} \in C_0(\C)$, so we only have to prove the other implication. Let $A := k_{0, (k)} \otimes k_{0, (k)}$ and $g := \Fc_W(A) = L_{k-1}^0(|\cdot|^2) e^{-\frac{|\cdot|^2}{2}}$. Then the convolution theorem shows
    \[\widetilde{T_{f, (k)}} = A \ast A \ast f = \left(\Fc_\sigma \Fc_\sigma (A \ast A)\right) \ast f = \Fc_\sigma(g) \ast \Fc_\sigma(g) \ast f.\]
    Let $X_0$ be the closed, $\alpha$-invariant subspace of $L^1(\C)$ generated by $\Fc_\sigma(g) \ast \Fc_\sigma(g)$. Then
    \[Z(X_0) := \set{\xi \in \C : \Fc_\sigma(h)(\xi) = 0 \text{ for all } h \in X_0} = \set{\xi \in \C : g(\xi)^2 = 0} = \Sigma_k.\]
    Now recall the functions $l_{0,k} = \Af_{k,1}k_{0,(1)}$ and define $B := l_{0,k} \otimes l_{0,k}$, $\varphi := \Fc_W(B) = e^{-\frac{|\cdot|^2}{2}}$. By Theorem \ref{thm:compactness}, it suffices to show that $B \ast T_{f, (k)} \in C_0(\C)$. Moreover,
    \[B \ast T_{f, (k)} = B \ast A \ast f = \left(\Fc_\sigma \Fc_\sigma (B \ast A)\right) \ast f = \Fc_\sigma(\varphi) \ast \Fc_\sigma(g) \ast f.\]
    Let $X_0'$ be the closed, $\alpha$-invariant subspace of $L^1(\C)$ generated by $\Fc_\sigma(\varphi) \ast \Fc_\sigma(g)$. Then
    \[Z(X_0') := \set{\xi \in \C : \Fc_\sigma(h)(\xi) = 0 \text{ for all } h \in X_0'} = \set{\xi \in \C : \varphi(\xi)g(\xi) = 0} = \Sigma_k.\]
    As $\Sigma_k$ is a set of spectral synthesis, we see that $X_0 = X_0'$ and therefore $\Fc_\sigma(\varphi) \ast \Fc_\sigma(g)$ can be approximated in $L^1(\C)$ by linear combinations of shifts of $\Fc_\sigma(g) \ast \Fc_\sigma(g)$. This implies that $B \ast T_{f, (k)}$ can be approximated in $C_0(\C)$ by linear combinations of shifts of $\widetilde{T_{f, (k)}}$ and is therefore itself in $C_0(\C)$. Therefore the assertion follows.
\end{proof}

As a final remark we note that the above argument of course works for any corresponding spaces $\Dc_1 \leftrightarrow \Dc_0$ in place of $\Kc(F^2_{(k)}) \leftrightarrow C_0(\C)$.

\section{Operator algebras}\label{sec:4}
In this section we will give several characterizations of $\Cc_1(F^2_n)$ and $\Cc_1(F^2_{(k)})$, $k,n \in \N$. This is in the same spirit as \cite[Theorem 1.1]{hagger2021}, except that we do not know whether $\Cc_1(F^2_n)$ is equal to the Toeplitz algebra for $n \geq 2$. Let us first introduce some of these algebras. Their definitions are exactly the same as in the analytic case.

\begin{definition}[Definition 1.1 of \cite{XiaZheng}]\label{def:suffloc}
Let $n \in \N$. An operator $A \in \Lc(F^2_n)$ is called \emph{sufficiently localized} if there are constants $C > 0$ and $\beta > 2$ such that
\[\abs{\sp{Ak_{z,n}}{k_{w,n}}} \leq \frac{C}{(1 + \abs{z-w})^{\beta}}\]
for all $w,z \in \C$. The set of sufficiently localized operators will be denoted by \nomenclature{$\Ac_{sl}$}{Sufficiently localized operators, p.~\pageref{def:suffloc}}$\Ac_{sl}(F^2_n)$.
\end{definition}

The composition of sufficiently localized operators is again sufficiently localized so that $\Ac_{sl}(F^2_n)$ is actually a $*$-algebra. This can be checked directly (see \cite[Proposition 3.2]{XiaZheng}), but also follows from our considerations below. As a consequence, the closure $\overline{\Ac}_{sl}(F^2_n)$ is in fact a $C^*$-algebra. In the following, $B(z,r)$ and $\overline{B}(z,r)$ denote open and closed balls of radius $r > 0$ around $z \in \C$, respectively.

\begin{definition}[Definition 1.1 of \cite{IsralowitzMitkovskiWick}]\label{def:weaklyloc}
Let $n \in \N$. An operator $A \in \Lc(F^2_n)$ is called \emph{weakly localized} if $T = A$ and $T = A^*$ both satisfy the conditions
\[\sup\limits_{z \in \C} \int_{\C} \abs{\sp{Tk_{z,n}}{k_{w,n}}} \, \mathrm{d}w < \infty, \quad \lim\limits_{r \to \infty} \sup\limits_{z \in \C} \int_{\C \setminus B(z,r)} \abs{\sp{Tk_{z,n}}{k_{w,n}}} \, \mathrm{d}w < \infty.\]
The set of weakly localized operators will be denoted by \nomenclature{$\Ac_{wl}$}{Weakly localized operators, p.~\pageref{def:weaklyloc}}$\Ac_{wl}(F^2_n)$.
\end{definition}

As for sufficiently localized operators, the weakly localized operators form a $*$-algebra (see \cite[Proposition 3.3]{IsralowitzMitkovskiWick}). Since all operators on $F^2_n$ can be represented as integral operators, both of these localization conditions can be understood as decay estimates of the corresponding kernel:
\begin{align*}
(Af)(w) &= \sp{Af}{K_{w,n}} = \sp{f}{A^*K_{w,n}} = \int_{\C} f(z)\sp{K_{z,n}}{A^*K_{w,n}} \, \mathrm{d}\mu(z)\\
&= \int_{\C} f(z)\sp{Ak_{z,n}}{k_{w,n}} e^{\frac{1}{2}(\abs{z}^2+\abs{w}^2)} \, \mathrm{d}\mu(z).
\end{align*}

In contrast to sufficiently and weakly localized operator that use the structure of $F^2_n$ as a reproducing kernel Hilbert space, our next algebra can be defined on any function space.

\begin{definition}[Definitions 3.2 and 3.9 of \cite{HaggerSeifert}]\label{def:bdo}
An operator $A \in \Lc(L^2(\C,\mu))$ is called a \emph{band operator} of band-width $\omega$ if
\[\omega := \sup\set{\dist(K,K') : K,K' \subseteq \C, M_{\1_{K'}}AM_{\1_K} \neq 0} < \infty,\]
where $\dist(K,K') := \inf\limits_{w \in K,z \in K'} \abs{w-z}$ denotes the distance between the sets $K$ and $K'$, and $M_{\1_K}$ is the multiplication by the characteristic function of $K$. The set of band operators will be denoted by $\BO$. Moreover, $A$ is called \emph{band-dominated} if \nomenclature{$\BDO$}{Band-dominated operators, p.~\pageref{def:bdo}}$A \in \overline{\BO} =: \BDO$, where the closure is taken in the operator norm. We will call $A \in \Lc(F^2_n)$ band-dominated if its extension $AP_n \from L^2(\C,\mu) \to L^2(\C,\mu)$ is band-dominated. The set of band-dominated operators on $F^2_n$ will be, with a slight abuse of notation, denoted by $P_n\BDO P_n$.
\end{definition}

If $A \in \Lc(L^2(\C,\mu))$ is an integral operator, then $A$ is a band operator if and only if its kernel $A(z,w)$ has band structure, that is, there exists $\omega \in \R$ such that $A(z,w) = 0$ for $\abs{z-w} > \omega$. On the other hand, a multiplication operator is the standard example of a band operator that is not an integral operator. Also note that for $A \in \Lc(F^2_n)$ the extension $AP_n$ is only a band operator if $A = 0$. In particular, if $A \in \Lc(F^2_n)$ is band-dominated and $(B_m)_{m \in \N}$ is a sequence of band operators converging to $AP_n$, then $B_m$ is usually not of the form $A_mP_n$ with $A_m \in \Lc(F^2_n)$. Hence the detour via $L^2(\C,\mu)$ is necessary for this algebra to make sense.

One of our goals in this section is to show that $\Cc_1(F^2_n)$ coincides with $\overline{\Ac_{sl}}(F^2_n)$, $\overline{\Ac_{wl}}(F^2_n)$ and $P_n\BDO P_n$. We first observe that sufficiently localized operators appear naturally in the framework of QHA. Let $g_t(z) := \frac{1}{\pi t}e^{-\frac{\abs{z}^2}{t}}$.

\begin{proposition} \label{prop:suff_loc}
For any $t > 0$ and $A \in \Lc(F^2_n)$ the operator $g_t \ast A$ is sufficiently localized.
\end{proposition}

To prove this proposition, we need the following simple lemma.

\begin{lemma} \label{lem:exp_integrals}
Let $c > 0$, $a,b \in \C$ and $p \in \C[z,\overline{z}]$ a polynomial of degree $m$. Then there is a polynomial $q_c \in \C[z,\overline{z}]$ of degree $m$ such that
\[\int_{\C} p(z,\overline{z})e^{az + b\overline{z} - c\abs{z}^2} \, \mathrm{d}z = \frac{\pi}{c}q_c(a,b) e^{\frac{ab}{c}}.\]
\end{lemma}

\begin{proof}
This follows directly from
\[\int_{\C} p(z,\overline{z})e^{az + b\overline{z} - c\abs{z}^2} \, \mathrm{d}z = \int_{\C} p\left(\tfrac{\partial}{\partial a},\tfrac{\partial}{\partial b}\right)e^{az + b\overline{z} - c\abs{z}^2} \, \mathrm{d}z = p\left(\tfrac{\partial}{\partial a},\tfrac{\partial}{\partial b}\right)\frac{\pi}{c}e^{\frac{ab}{c}}.\qedhere\]
\end{proof}

\begin{proof}[Proof of Proposition \ref{prop:suff_loc}]
For $x,y \in \C$ we have
\begin{align*}
\sp{(g_t \ast A)k_{x, n}}{k_{y, n}} &= \int_{\C} g_t(z) \sp{AW_{-z}k_{x, n}}{W_{-z}k_{y, n}} \, \mathrm{d}z\\
&= \int_{\C} g_t(z) e^{\frac{1}{2}(\overline{x}z - x\overline{z} + y\overline{z} - \overline{y}z)}\sp{Ak_{x-z, n}}{k_{y-z, n}} \, \mathrm{d}z\\
&= \int_{\C} g_t(z) e^{\frac{1}{2}(\overline{x}z - x\overline{z} + y\overline{z} - \overline{y}z)} \frac{1}{\pi} \int_{\C} (Ak_{x-z, n})(w)\overline{k_{y-z, n}(w)} e^{-\abs{w}^2} \, \mathrm{d}w \, \mathrm{d}z\\
&= \frac{1}{\pi} \int_{\C} \int_{\C} g_t(z) e^{\frac{1}{2}(\overline{x}z - x\overline{z} + y\overline{z} - \overline{y}z)} \frac{1}{\pi}\int_{\C} k_{x-z, n}(v)\overline{(A^*K_{w,n})(v)} e^{-\abs{v}^2} \, \mathrm{d}v\\
&\qquad \qquad \quad \times \overline{k_{y-z, n}(w)} e^{-\abs{w}^2} \, \mathrm{d}w \, \mathrm{d}z\\
&= \frac{1}{n\pi^3t} \int_{\C} \int_{\C} \int_{\C} e^{-\frac{\abs{z}^2}{t}} e^{\frac{1}{2}(\overline{x}z - x\overline{z} + y\overline{z} - \overline{y}z)} L_{n-1}^1(\abs{v-x+z}^2)e^{v(\overline{x}-\overline{z}) - \frac{1}{2}\abs{x-z}^2}\\
&\qquad \times L_{n-1}^1(\abs{w-y+z}^2)e^{(y-z)\overline{w} - \frac{1}{2}\abs{y-z}^2} \overline{(A^*K_{w,n})(v)} e^{-\abs{v}^2-\abs{w}^2} \, \mathrm{d}v \, \mathrm{d}w \, \mathrm{d}z\\
&= \frac{1}{n\pi^3t} \int_{\C} \int_{\C} \int_{\C} e^{-\frac{t+1}{t}\abs{z}^2} e^{(\overline{x} - \overline{w})z} e^{(y - v)\overline{z}} L_{n-1}^1(\abs{v-x+z}^2)\\
&\qquad \times L_{n-1}^1(\abs{w-y+z}^2) e^{v\overline{x} - \frac{1}{2}\abs{x}^2 + y\overline{w} - \frac{1}{2}\abs{y}^2 - \abs{v}^2 - \abs{w}^2}\overline{(A^*K_{w,n})(v)} \, \mathrm{d}v \, \mathrm{d}w \, \mathrm{d}z\\
&= \frac{1}{n\pi^3t} \int_{\C} \int_{\C} \int_{\C} e^{-\frac{t+1}{t}\abs{z}^2} e^{(\overline{x} - \overline{w} - \overline{y})z} e^{(y - v - x)\overline{z}} L_{n-1}^1(\abs{v+z}^2)L_{n-1}^1(\abs{w+z}^2)\\
&\qquad \times e^{(v+x)\overline{x} - \frac{1}{2}\abs{x}^2 + y(\overline{w}+\overline{y}) - \frac{1}{2}\abs{y}^2 - \abs{v+x}^2 - \abs{w+y}^2}\overline{(A^*K_{w+y,n})(v+x)} \, \mathrm{d}v \, \mathrm{d}w \, \mathrm{d}z\\
&= \frac{1}{n\pi^3t} \int_{\C} \int_{\C} \int_{\C} e^{-\frac{t+1}{t}\abs{z}^2} e^{(\overline{x} - \overline{w} - \overline{y})z} e^{(y - v - x)\overline{z}} L_{n-1}^1(\abs{v+z}^2)L_{n-1}^1(\abs{w+z}^2)\\
&\qquad \qquad \qquad \quad e^{-x\overline{v} - \frac{1}{2}\abs{x}^2 - w\overline{y} - \frac{1}{2}\abs{y}^2 - \abs{v}^2 - \abs{w}^2}\overline{(A^*K_{w+y,n})(v+x)} \, \mathrm{d}v \, \mathrm{d}w \, \mathrm{d}z.
\end{align*}
For the $z$-integral we get
\begin{align*}
&\int_{\C} e^{-\frac{t+1}{t}\abs{z}^2} e^{(\overline{x} - \overline{w} - \overline{y})z} e^{(y - v - x)\overline{z}} L_{n-1}^1(\abs{v+z}^2)L_{n-1}^1(\abs{w+z}^2) \, \mathrm{d}z\\
&\qquad \qquad \qquad \qquad \qquad \qquad \qquad \qquad = \frac{\pi t}{t+1}q(\overline{x}-\overline{y}-\overline{w},y-v-x)e^{\frac{t}{t+1}(\overline{x}-\overline{y}-\overline{w})(y-v-x)}
\end{align*}
for some polynomial $q$ by Lemma \ref{lem:exp_integrals}. Note that the coefficients of $q$ depend polynomially on $v$, $w$ and their complex conjugates. We can therefore choose a polynomial $\tilde{q}$ such that
\[\abs{q(\overline{x}-\overline{y}-\overline{w},y-v-x)} \leq \tilde{q}(v,\overline v,w,\overline w,x-y,\overline x - \overline y)\]
for all $v,w,x,y \in \C$. It follows
\begin{align*}
\abs{\sp{(g_t \ast A)k_{x,n}}{k_{y,n}}} &\leq \frac{1}{n\pi^2(t+1)} \int_{\C} \int_{\C} \tilde{q}(v,\overline v,w,\overline w,x-y,\overline x - \overline y)\abs{e^{\frac{t}{t+1}(\overline{x}-\overline{y}-\overline{w})(y-v-x)}}\\
&\qquad \qquad \qquad \quad \times \abs{e^{-x\overline{v} - \frac{1}{2}\abs{x}^2 - w\overline{y} - \frac{1}{2}\abs{y}^2 - \abs{v}^2 - \abs{w}^2}(A^*K_{w+y,n})(v+x)} \mathrm{d}v \, \mathrm{d}w\\
&\leq \frac{\norm{A}}{\pi^2(t+1)} \int_{\C} \int_{\C} \tilde{q}(v,\overline v,w,\overline w,x-y,\overline x - \overline y)\abs{e^{\frac{t}{t+1}(\overline{x}-\overline{y}-\overline{w})(y-v-x)}}\\
&\qquad \qquad \qquad \quad \times \abs{e^{-x\overline{v} - \frac{1}{2}\abs{x}^2 - w\overline{y} - \frac{1}{2}\abs{y}^2 - \abs{v}^2 - \abs{w}^2}}e^{\frac{1}{2}\abs{w+y}^2+\frac{1}{2}\abs{v+x}^2} \, \mathrm{d}v \, \mathrm{d}w\\
&= \frac{\norm{A}}{\pi^2(t+1)} \int_{\C} \int_{\C} \tilde{q}(v,\overline v,w,\overline w,x-y,\overline x - \overline y)e^{-\frac{t}{t+1}\Re\left((\overline{x}-\overline{y}-\overline{w})v\right) - \frac{1}{2}\abs{v}^2}\\
&\qquad \qquad \qquad \quad \times e^{\frac{t}{t+1}\Re\left((\overline{x}-\overline{y})w\right) - \frac{1}{2}\abs{w}^2} e^{-\frac{t}{t+1}\abs{x-y}^2} \, \mathrm{d}v \, \mathrm{d}w\\
&= \frac{2\norm{A}}{\pi(t+1)} \int_{\C} p(w,\overline w,x-y,\overline x - \overline y)e^{\frac{t^2}{2(t+1)^2}\abs{x-y-w}^2} e^{\frac{t}{t+1}\Re\left((\overline{x}-\overline{y})w\right) - \frac{1}{2}\abs{w}^2}\\
&\qquad \qquad \qquad \quad \times e^{-\frac{t}{t+1}\abs{x-y}^2} \, \mathrm{d}w
\end{align*}
for some polynomial $p$ again by Lemma \ref{lem:exp_integrals}.  We conclude
\begin{align*}
\abs{\sp{(g_t \ast A)k_x}{k_y}} &\leq \frac{2\norm{A}}{\pi(t+1)} \int_{\C} p(w,\overline w,x-y,\overline x - \overline y)e^{-\frac{2t+1}{2(t+1)^2}\abs{w}^2} e^{\frac{t}{(t+1)^2}\Re\left((\overline{x}-\overline{y})w\right)}\\
&\qquad \qquad \qquad \quad \times e^{-\frac{t(t+2)}{2(t+1)^2}\abs{x-y}^2} \, \mathrm{d}w\\
&= \frac{4(t+1)\norm{A}}{2t+1}r(x-y,\overline{x}-\overline{y}) e^{\frac{t^2}{2(t+1)^2(2t+1)}\abs{x-y}^2} e^{-\frac{t(t+2)}{2(t+1)^2}\abs{x-y}^2}\\
&= \frac{4(t+1)\norm{A}}{2t+1}r(x-y,\overline{x}-\overline{y})e^{-\frac{t}{2t+1}\abs{x-y}^2}
\end{align*}
for yet another polynomial $r$. This shows that $g_t \ast A$ is sufficiently localized for any $t > 0$ and $A \in \Lc(F^2_n)$.
\end{proof}

Together with Lemma \ref{lem:approximation} we immediately obtain the following corollary.

\begin{corollary}
$\Cc_1(F^2_n) \subseteq \overline{\Ac}_{sl}(F^2_n)$.
\end{corollary}

Our next goal is to show that $P_n\BDO P_n \subseteq \Cc_1(F^2_n)$. The following lemma is similar to \cite[Lemma 4.19]{Bauer_Fulsche2020}. It shows that the normalized reproducing kernels are uniformly, exponentially localized.

\begin{lemma} \label{lem:kernel_localization}
For each $n \in \N$ there exists a constant $C_n$ such that
\[\norm{M_{1 - \chi_{B(z,R)}}k_{z,n}} \leq C_n e^{-\frac{R^2}{4}}\]
for all $z \in \C$ and $R > 0$.
\end{lemma}

\begin{proof}
As $k_{z,n} = W_zk_{0,n}$ and $M_{1 - \chi_{B(z,R)}} = W_zM_{1 - \chi_{B(0,R)}}W_{-z}$, it suffices to check the lemma for $z = 0$. Since $k_{0,n}(z) = \frac{1}{\sqrt{n}}L_{n-1}^1(\abs{z}^2)$, we get
\[\norm{M_{1 - \chi_{B(0,R)}}k_{0,n}}^2 = \frac{1}{n} \int_{\C \setminus B(0,R)} L_{n-1}^1(\abs{z}^2)^2 \, \mathrm{d}\mu(z) = \frac{2}{n}\int_R^{\infty} L_{n-1}^1(r^2)^2 re^{-r^2} \, \mathrm{d}r.\]
Using $L_{n-1}^1(r^2)^2 \leq \frac{nC_n^2}{2}e^{\frac{r^2}{2}}$ for some constant $C_n \geq 0$, we obtain
\[\norm{M_{1 - \chi_{B(0,R)}}k_{0,n}} \leq C_n\left(\int_R^{\infty} re^{-\frac{r^2}{2}} \, \mathrm{d}r\right)^{\frac{1}{2}} = C_ne^{-\frac{R^2}{4}}\]
as expected.
\end{proof}

Next we give a criterion for certain integral operators $A \from L^2(\C,\mu) \to L^2(\C,\mu)$ to satisfy the $\Cc_1$-condition $\lim\limits_{\abs{z} \to 0} \norm{\alpha_z(A)-A} = 0$. Such operators will then be used to approximate operators of the form $P_nAP_n$ with $A \in \BDO$.

\begin{lemma} \label{lem:BO_with_BUC_kernel}
Let $g \from \C^2 \to \C$ be a bounded measurable function such that
\begin{itemize}
	\item[(i)] $\set{z \mapsto g(x-z,y-z) : x,y \in \C}$ is equicontinuous at $z = 0$,
	\item[(ii)] there is an $\omega \geq 0$ such that $g(x,y) = 0$ for all $x,y \in \C$ with $\abs{x-y} \geq \omega$.
\end{itemize}
Define the integral operator $A \from L^2(\C,\mu) \to L^2(\C,\mu)$ by
\[(Af)(x) := e^{\frac{\abs{x}^2}{2}}\int_{\C} g(x,y)f(y) e^{\frac{\abs{y}^2}{2}} \, \mathrm{d}\mu(y).\]
Then $A$ is a band operator of band-width at most $\omega$ and $\norm{\alpha_z(A)-A} \to 0$ as $\abs{z} \to 0$.
\end{lemma}

\begin{proof}
We have
\begin{align*}
(W_zAW_{-z}f)(x) &= e^{x\overline z - \frac{\abs{z}^2}{2}}AW_{-z}f(x-z)\\
&= e^{i\Im(x\overline z) + \frac{\abs{x}^2}{2}}\int_{\C} g(x-z,y)W_{-z}f(y) e^{\frac{\abs{y}^2}{2}} \, \mathrm{d}\mu(y)\\
&= \frac{1}{\pi}e^{i\Im(x\overline z) + \frac{\abs{x}^2}{2}}\int_{\C} g(x-z,y) f(y+z) e^{-\frac{\abs{y+z}^2}{2}} e^{-i\Im(y\overline z)} \, \mathrm{d}y\\
&= \frac{1}{\pi}e^{i\Im(x\overline z) + \frac{\abs{x}^2}{2}}\int_{\C} g(x-z,y-z) f(y) e^{-\frac{\abs{y}^2}{2}} e^{-i\Im(y\overline z)} \, \mathrm{d}y\\
&= \frac{1}{\pi}e^{i\Im(x\overline z) + \frac{\abs{x}^2}{2}}\int_{B(x,\omega)} g(x-z,y-z) f(y) e^{-\frac{\abs{y}^2}{2}} e^{-i\Im(y\overline z)} \, \mathrm{d}y.
\end{align*}
Thus
\[\norm{(\alpha_z(A)-A)f}^2 = \frac{1}{\pi^3}\int_{\C} \abs{\int_{B(x,\omega)} \left(g(x-z,y-z) e^{i\Im((x-y)\overline z)}  - g(x,y)\right) f(y) e^{-\frac{\abs{y}^2}{2}} \, \mathrm{d}y}^2 \, \mathrm{d}x.\]
Let $\epsilon > 0$. As $\set{z \mapsto g(x-z,y-z) e^{i\Im((x-y)\overline z)} : \abs{x-y} < \omega}$ is equicontinuous at $z = 0$, we may choose $\delta > 0$ such that $\abs{g(x-z,y-z) e^{i\Im((x-y)\overline z)}  - g(x,y)} < \epsilon$ for $\abs{z} < \delta$ and $\abs{x-y} < \omega$. By H\"older's inequality and Fubini, we obtain
\begin{align*}
\norm{(\alpha_z(A)-A)f}^2 &\leq C\epsilon^2\int_{\C} \int_{B(x,\omega)} \abs{f(y)}^2 \, \mathrm{d}\mu(y) \, \mathrm{d}x\\
&= C\epsilon^2\int_{\C} \int_{B(y,\omega)} \abs{f(y)}^2 \, \mathrm{d}x \, \mathrm{d}\mu(y)\\
&= \tilde{C}\epsilon^2\norm{f}^2
\end{align*}
for some constants $C$, $\tilde{C}$ that only depend on $\omega$. This shows $\norm{\alpha_z(A)-A} \to 0$ for $\abs{z} \to 0$.

The boundedness of $A$ follows similarly and then $A \in \BO$ is clear.
\end{proof}

Given an arbitrary bounded operator $A \from L^2(\C,\mu) \to L^2(\C,\mu)$ we now construct an approximation $(A_R)_{R > 0}$ of $P_nAP_n$ that satisfies the assumptions of Lemma \ref{lem:BO_with_BUC_kernel}. In Proposition \ref{prop:BDO_in_C_1} below we then prove that $\lim\limits_{R \to \infty} \norm{P_nAP_n-A_R} = 0$ in case $A$ is a band operator. This will be enough to conclude that $P_n\BDO P_n \subseteq \Cc_1(F^2_n)$.

\begin{proposition} \label{prop:band_operators}
Let $A \from L^2(\C,\mu) \to L^2(\C,\mu)$ be a bounded linear operator and $R > 0$. Consider the integral operator $A_R \from L^2(\C,\mu) \to L^2(\C,\mu)$ with kernel $\sp{AK_{n,y}}{K_{n,x}}\chi_{B(0,R)}(x-y)$. Then $A_R$ is a band operator of band-width at most $R$ and $\norm{\alpha_z(A_R)-A_R} \to 0$ as $\abs{z} \to 0$.
\end{proposition}

\begin{proof}
Let $g_R(x,y) := \sp{Ak_{y,n}}{k_{x,n}}\chi_{B(0,R)}(x-y)$. In order to apply Lemma \ref{lem:BO_with_BUC_kernel}, we need to show that $\set{z \mapsto g_R(x-z,y-z) : x,y \in \C}$ is equicontinuous at $z = 0$. Using \eqref{eq:shifted_kernels}, we obtain
\begin{align*}
g_R(x-z,y-z) &= \sp{Ak_{y-z,n}}{k_{x-z,n}}\chi_{B(0,R)}(x-y)\\
&= e^{i\Im((x-y)\overline z)}\sp{W_{-x}AW_yk_{-z,n}}{k_{-z,n}}\chi_{B(0,R)}(x-y).
\end{align*}
Clearly, $\set{z \mapsto e^{i\Im((x-y)\overline z)}\chi_{B(0,R)}(x-y) : x,y \in \C}$ is bounded and equicontinuous at $z = 0$. It thus remains to consider the inner product. For this we obtain
\begin{align*}
&\abs{\sp{W_{-x}AW_yk_{-z,n}}{k_{-z,n}} - \sp{W_{-x}AW_yk_{0,n}}{k_{0,n}}}\\
&\qquad \leq \abs{\sp{W_{-x}AW_y(k_{-z,n} - k_{0,n})}{k_{-z,n}}} + \abs{\sp{W_{-x}AW_yk_{0,n}}{k_{-z,n}-k_{0,n}}}\\
&\qquad \leq 2\norm{A}\norm{k_{-z,n}-k_{0,n}}.
\end{align*}
The map $z \mapsto k_{-z,n}(w)$ is obviously continuous for all $w \in \C$ and therefore Scheff\'e's lemma implies the equicontinuity of $\set{z \mapsto \sp{W_{-x}AW_yk_{-z,n}}{k_{-z,n}} : x,y \in \C}$. It follows that the set $\set{z \mapsto g_R(x-z,y-z) : x,y \in \C}$ is equicontinuous at $z = 0$ and thus $\norm{\alpha_z(A_R)-A_R} \to 0$ as $\abs{z} \to 0$ by Lemma \ref{lem:BO_with_BUC_kernel}.
\end{proof}

\begin{proposition} \label{prop:BDO_in_C_1}
$P_n\BDO P_n \subseteq \Cc_1(F^2_n)$.
\end{proposition}

\begin{proof}
It suffices to show that if $A$ is a band operator, then
\[\norm{\alpha_z(P_nAP_n) - P_nAP_n} = \norm{P_n(\alpha_z(A)-A)P_n} \to 0\]
as $\abs{z} \to 0$. So let $A$ be a band operator with band-width at most $\omega$. We note that since
\[(P_nAP_nf)(x) = \sp{AP_nf}{K_{x,n}} = \sp{f}{P_nA^*K_{x,n}},\]
$P_nAP_n$ is an integral operator with kernel 
\[\overline{P_nA^*K_{x,n}(y)} = \sp{AK_{y,n}}{K_{x,n}} = n\sp{Ak_{y,n}}{k_{x,n}}e^{\frac{\abs{y}^2}{2}}e^{\frac{\abs{x}^2}{2}}.\]
Applying Lemma \ref{lem:kernel_localization} with $R := \frac{1}{3}\abs{x-y}$, we obtain
\begin{align*}
\abs{\sp{Ak_{y,n}}{k_{x,n}}} &= |\langle AM_{1 - \chi_{B(y,r)}}k_{y,n} , k_{x,n} \rangle| + |\langle AM_{\chi_{B(y,r)}}k_{y,n} , M_{1 - \chi_{B(y,r)}}k_{x,n} \rangle|\\
&\leq 2\norm{A}C_ne^{-\frac{\abs{x-y}^2}{36}}  
\end{align*}
for $\abs{x-y} > 3\omega$. For $R > 3\omega$ let $A_R \from L^2(\C,\mu) \to L^2(\C,\mu)$ be the integral operator with kernel
\[\sp{AK_{y,n}}{K_{x,n}}\chi_{B(0,R)}(x-y) =n \sp{Ak_{y,n}}{k_{x,n}}e^{\frac{\abs{x}^2}{2}}e^{\frac{\abs{y}^2}{2}}\chi_{B(0,R)}(x-y).\]
Proposition \ref{prop:band_operators} implies that $\norm{\alpha_z(A_R)-A_R} \to 0$ as $\abs{z} \to 0$. It therefore remains to show $\lim\limits_{R \to \infty} \norm{P_nAP_n - A_R} = 0$:
\begin{align*}
\norm{(P_nAP_n - A_R)f}^2 &= \frac{n^2}{\pi^3}\int_{\C} \abs{\int_{\C} \sp{Ak_{y,n}}{k_{x,n}}(1 - \chi_{B(0,R)}(x-y))f(y)e^{\frac{-\abs{y}^2}{2}} \, \mathrm{d}y}^2 \, \mathrm{d}x\\
&\leq C\int_{\C} \left(\int_{\C} e^{-\frac{\abs{x-y}^2}{36}}(1 - \chi_{B(0,R)}(x-y))\abs{f(y)}e^{\frac{-\abs{y}^2}{2}} \, \mathrm{d}y\right)^2 \, \mathrm{d}x\\
&\leq \tilde{C}\left(\int_{\C \setminus B(0,R)} e^{-\frac{\abs{x}^2}{36}} \, \mathrm{d}x\right)^2 \norm{f}^2
\end{align*}
for some constants $C$ and $\tilde{C}$ by Young's inequality.
\end{proof}

As $\overline{\Ac}_{sl}(F^2_n) \subseteq \overline{\Ac}_{wl}(F^2_n)$ is clear, it remains to show $\overline{\Ac}_{wl}(F^2_n) \subseteq P_n\BDO P_n$ in order to prove the equality of all the algebras introduced above. The idea is to use the same approximation as in Proposition \ref{prop:band_operators}.

\begin{proposition}
$\overline{\Ac}_{wl}(F^2_n) \subseteq P_n\BDO P_n$.
\end{proposition}

\begin{proof}
Let $A \in \Ac_{wl}$. As in Proposition \ref{prop:BDO_in_C_1} we can write $A$ as an integral operator with kernel $n\sp{Ak_{y,n}}{k_{x,n}}e^{\frac{\abs{y}^2}{2}}e^{\frac{\abs{x}^2}{2}}$. For $R > 0$ let $A_R$ be the integral operator with kernel
\[n\sp{Ak_{y,n}}{k_{x,n}}e^{\frac{\abs{y}^2}{2}}e^{\frac{\abs{x}^2}{2}}\chi_{B(0,R)}(x-y).\]
Proposition \ref{prop:band_operators} implies that $A_R$ is a band operator. It remains to show that $\norm{A-A_R} \to 0$ as $R \to \infty$. H\"older's inequality and Fubini imply
\begin{align*}
\abs{\sp{(A - A_R)f}{g}} &\leq \frac{n}{\pi^2}\int_{\C} \int_{\C} \abs{\sp{Ak_{y,n}}{k_{x,n}}}(1 - \chi_{B(0,R)}(x-y))\abs{f(y)}e^{\frac{-\abs{y}^2}{2}} \abs{g(x)}e^{\frac{-\abs{x}^2}{2}} \, \mathrm{d}y \, \mathrm{d}x\\
&\leq \frac{n}{\pi^2}\left(\int_{\C} \int_{\C} \abs{\sp{Ak_{y,n}}{k_{x,n}}}(1 - \chi_{B(0,R)}(x-y))\abs{f(y)}^2e^{-\abs{y}^2} \, \mathrm{d}y \, \mathrm{d}x\right)^{1/2}\\
&\qquad \times \left(\int_{\C} \int_{\C} \abs{\sp{Ak_{y,n}}{k_{x,n}}}(1 - \chi_{B(0,R)}(x-y))\abs{g(x)}^2e^{-\abs{x}^2} \, \mathrm{d}y \, \mathrm{d}x\right)^{1/2}\\
&= \frac{n}{\pi^2}\left(\int_{\C} \int_{\C \setminus B(0,R)} \abs{\sp{Ak_{y,n}}{k_{x,n}}} \, \mathrm{d}x \abs{f(y)}^2e^{-\abs{y}^2} \, \mathrm{d}y\right)^{1/2}\\
&\qquad \times \left(\int_{\C} \int_{\C \setminus B(0,R)} \abs{\sp{Ak_{y,n}}{k_{x,n}}} \, \mathrm{d}y \abs{g(x)}^2e^{-\abs{x}^2} \, \mathrm{d}x\right)^{1/2}\\
&\leq \frac{n}{\pi}\left(\sup\limits_{y \in \C} \int_{\C \setminus B(0,R)} \abs{\sp{Ak_{y,n}}{k_{x,n}}} \, \mathrm{d}x\right)^{1/2}\norm{f}\\
&\qquad \times \left(\sup\limits_{x \in \C} \int_{\C \setminus B(0,R)} \abs{\sp{Ak_{y,n}}{k_{x,n}}} \, \mathrm{d}y\right)^{1/2}\norm{g}
\end{align*}
and therefore the assertion follows.
\end{proof}

Combining all the above results, we arrive at the main theorem of this section. 

\begin{theorem} \label{thm:algebras}
We have
\[\Cc_1(F^2_n) = \overline{\Ac}_{sl}(F^2_n) = \overline{\Ac}_{wl}(F^2_n) = P_n\BDO P_n.\]
\end{theorem}

The same result can also be obtained for the true polyanalytic Fock spaces, either by applying the same arguments as above or simply by applying the projection $P_{(k)}$. We state the result here as a corollary for later reference.

\begin{corollary}
Let $j,k \in \N$. Then
\[\Cc_1(F^2_{(j)},F^2_{(k)}) = \overline{\Ac}_{sl}(F^2_{(j)},F^2_{(k)}) = \overline{\Ac}_{wl}(F^2_{(j)},F^2_{(k)}) = P_{(k)}\BDO P_{(j)},\]
where the above sets of operators are defined as for $F^2_n$ with the obvious modifications.
\end{corollary}

As in the analytic case, there is another characterization of $\Cc_1(F^2_{(k)})$ involving essential commutants. For a set $\Sc$ of bounded linear operators on a Hilbert space $\Hc$, the essential commutant of $\Sc$ is defined as
\[\EssCom(\Sc) := \set{B \in \Lc(\Hc) : AB - BA \in \Kc(\Hc) \text{ for all } A \in \Sc}.\]
We will use the abbreviation $[A,B] := AB-BA$ for the commutator. Moreover, let $\VO(\C)$ denote the set of functions $f \from \C \to \C$ with vanishing oscillation, that is,
\[\lim\limits_{\abs{z} \to \infty} \sup\limits_{w \in \overline{B}(z,1)} \abs{f(z)-f(w)} = 0.\]

\begin{theorem}\label{thm:esscom_new}
    For every $k \in \mathbb N$ we have
    \begin{align} \label{eq:esscom_new}
        \mathfrak A_{k,1} \overline{\{ T_{f, (1)}: f \in \VO(\mathbb C)\}} \mathfrak A_{1,k} = \overline{\{ T_{f, (k)}: f \in \VO(\mathbb C)\}}.
    \end{align}
    In particular, 
    \begin{align*}
    \Cc_1(F^2_{(k)}) &= \EssCom(\set{T_{f,(k)} : f \in \VO(\C)}),\\
    \EssCom(\Cc_1(F^2_{(k)})) &= \set{T_{f,(k)} + K : f \in \VO(\C), K \in \Kc(F^2_{(k)})}\\
    &= \overline{\{ T_{f, (k)}: ~f \in \VO(\mathbb C)\}}.
    \end{align*}
\end{theorem}

\begin{proof}
    Let $f \in \VO(\mathbb C)$. As both $T_{f, (1)}$ and $T_{f, (k)}$ are compressions of $T_{f,n}$, a combination of Theorem 14 and Lemma 19 of \cite{hagger2023} shows that $T_{f, (k)} - \mathfrak A_{k,1} T_{f, (1)} \mathfrak A_{1,k}$ is compact. Since $k_{0, (k)} \otimes k_{0, (k)}$ is $\infty$-regular, we know from Theorem \ref{thm:inftyregular} that $\mathcal K(F_{(k)}^2) \subseteq \overline{\{ T_{f, (k)}: f \in \VO(\mathbb C)\}}$ for every $k \in \mathbb N$. Therefore \eqref{eq:esscom_new} follows. The other identities are now carried over from the case $k = 1$, which is established in \cite{hagger2021}.
\end{proof}

We also have an analogous result for $F^2_n$, but we do not know whether the last equality of Theorem \ref{thm:esscom_new} is also true.

\begin{theorem}\label{thm:esscom_n}
    For every $n \in \mathbb N$ we have
    \begin{align*}
    \Cc_1(F^2_n) &= \EssCom(\set{T_{f,n} : f \in \VO(\C)}),\\
    \EssCom(\Cc_1(F^2_n)) &= \set{T_{f,n} + K : f \in \VO(\C), K \in \Kc(F^2_n)}.
    \end{align*}
\end{theorem}

\begin{proof}
In the following, we will consider operators $A \in \Lc(F^2_n)$ as matrices with respect to the decomposition $F^2_n = \bigoplus\limits_{k = 1}^n F^2_{(k)}$. In particular, $A_{k,j} = P_{(k)}A|_{F^2_{(j)}} \in \Lc(F^2_{(j)},F^2_{(k)})$. With this notation we have $A \in \EssCom(\set{T_{f,n} : f \in \VO(\C)})$ if and only if
\begin{align*}
    \sum_{m=1}^n A_{k,m}B_{m,j} - B_{k,m}A_{m,j} \in \mathcal K(F_{(j)}^2, F_{(k)}^2)
\end{align*}
for every $j,k \in \set{1,\ldots,n}$ and every $B = T_{f, n}$ with $f \in \VO(\mathbb C)$. Now observe that $B = T_{f,n}$ is essentially diagonal if $f \in \VO(\mathbb C)$, that is, $B_{k,j}$ is compact for $j \neq k$. This follows directly from the fact that $B_{k,j}$ is the compression of a Hankel operator with $\VO$-symbol and therefore compact (see \cite[Theorem 30]{hagger2023}). Hence, $A \in \EssCom(\set{T_{f,n} : f \in \VO(\C)})$ if and only if
\begin{align*}
    A_{k,j}B_{j,j} - B_{k,k}A_{k,j} \in \mathcal K(F_{(j)}^2, F_{(k)}^2)
\end{align*}
for all $j,k \in \set{1,\ldots,n}$. Now, $B_{j,j} = T_{f, (j)}$ and $B_{k,k} = T_{f, (k)} = \mathfrak A_{k,j} T_{f, (j)} \mathfrak A_{j,k} + K$ for some compact operator $K$ by Theorem \ref{thm:esscom_new}. Hence, we see that $A \in \EssCom(\set{T_{f,n} : f \in \VO(\C)})$ if and only if $\mathfrak A_{j, k}A_{k,j} \in \EssCom \set{T_{f, (j)}: f \in \VO(\mathbb C)}$ for all $j,k \in \set{1,\ldots,n}$. By Theorem \ref{thm:esscom_new}, this is equivalent to $\mathfrak A_{j, k}A_{k,j} \in \mathcal C_1(F_{(j)})$, which is again equivalent to $A_{k,j} \in \mathcal C_1(F_{(j)}^2, F_{(k)}^2)$. This shows the first equality.

To obtain the second equality we only need to prove that
\[\EssCom(\Cc_1(F^2_n)) \subseteq \set{T_{f,n} + K : f \in \VO(\C), K \in \Kc(F^2_n)}\]
as the other inclusion is clear from the first equality. So let $B \in \EssCom(\Cc_1(F_n^2))$. In particular, $B$ essentially commutes with the projections $P_{(k)}$ and therefore $B$ has to be essentially diagonal. Moreover, we know that $B_{k,k} \in T_{f_k, (k)} + \mathcal K(F_{(k)}^2)$ for some $f_k \in \VO(\mathbb C)$ by Theorem \ref{thm:esscom_new}. As $\Af_{1,k} \in \Cc_1(F^2_n)$, $B$ essentially commutes with $\Af_{1,k}$ and we get that $T_{f_1, (1)}\Af_{1,k} - \Af_{1,k}T_{f_k, (k)}$ is compact. But $T_{f_1,(k)} - \Af_{k,1}T_{f_1,(1)}\Af_{1,k}$ is compact as well by Theorem \ref{thm:esscom_new}. It follows that $T_{f_1,(k)} - T_{f_k,(k)}$ is also compact. Hence $B = T_{f_1,n} + K$ for some $K \in \mathcal K(F_{(k)}^2)$.
\end{proof}

Finally, we remark that on $L^2(\C,\mu)$ the equality $\Cc_1 = \BDO$ no longer holds (and it is not even clear how to define $\Ac_{sl}$ or $\Ac_{wl}$ on $L^2(\C,\mu)$).

\begin{example}~
\begin{itemize}
	\item[(a)] Let $f \in L^{\infty}(\C)$. Then obviously $M_f \in \BDO$, but $M_f \notin \Cc_1(L^2(\C,\mu))$ unless $f \in \BUC(\C)$.
	\item[(b)] Let $g$ be a Lipschitz-continuous function with support in $B(0,1)$ and consider the operator $A \from L^2(\C,\mu) \to L^2(\C,\mu)$ defined by
	\[(Af)(x) := e^{\frac{\abs{x}^2}{2}}\int_{\C} g(x-2y)e^{i\Im(x\overline y)}f(y) e^{\frac{\abs{y}^2}{2}} \, \mathrm{d}\mu(y).\]
	Then a similar computation as in Lemma \ref{lem:BO_with_BUC_kernel} shows that
	\begin{align*}
	&\norm{(\alpha_z(A) - A)f}^2\\
	&\qquad = \frac{1}{\pi^3}\int_{\C} \abs{\int_{B(\frac{x}{2},\frac{1+\abs{z}}{2})} \left(g(x-2y+z) - g(x-2y)\right) e^{i\Im(x\overline y)} f(y) e^{-\frac{\abs{y}^2}{2}} \, \mathrm{d}y}^2 \, \mathrm{d}x\\
	&\qquad \leq C\abs{z}^2\int_{\C} \left(\int_{B(\frac{x}{2},\frac{1+\abs{z}}{2})} \abs{f(y)} e^{-\frac{\abs{y}^2}{2}} \, \mathrm{d}y\right)^2 \, \mathrm{d}x\\
	&\qquad \leq \tilde{C}\abs{z}^2\int_{\C} \int_{B(\frac{x}{2},\frac{1+\abs{z}}{2})} \abs{f(y)}^2 e^{-\abs{y}^2} \, \mathrm{d}y \, \mathrm{d}x\\
	&\qquad = \tilde{C}\abs{z}^2\int_{\C} \int_{B(2y,1+\abs{z})} \abs{f(y)}^2 e^{-\abs{y}^2} \, \mathrm{d}x \, \mathrm{d}y\\
	&\qquad = \hat{C}\abs{z}^2\norm{f}^2
	\end{align*}
	for some constants $C$, $\tilde{C}$ and $\hat{C}$. Thus $A \in \Cc_1$. On the other hand, if $f_1(y) := e^{i\Im(y\overline w)}k_{w,1}(y)$ and $f_2(x) := e^{-i\Im(x\overline w)}k_{2w,1}(x)$, then $\norm{f_1} = \norm{f_2} = 1$ and
	\begin{align*}
	&\sp{A(\chi_{B(w,1)}f_1)}{\chi_{B(2w,1)}f_2}\\
	&= \quad \frac{1}{\pi^2} \int_{B(2w,1)} \int_{B(w,1)} g(x-2y)e^{i\Im(x\overline y + y\overline w + x\overline w)}e^{y\overline w + 2\overline xw - \frac{5}{2}\abs{w}^2 - \frac{1}{2}\abs{x}^2 - \frac{1}{2}\abs{y}^2} \, \mathrm{d}y \, \mathrm{d}x\\
	&= \quad \frac{1}{\pi^2} \int_{B(0,1)} \int_{B(0,1)} g(x-2y) e^{-\frac{1}{2}\abs{x}^2 - \frac{1}{2}\abs{y}^2 + i\Im(x\overline y)} \, \mathrm{d}y \, \mathrm{d}x.
	\end{align*}
	The distance of $A$ to any band operator is at least the value of this double integral. If we choose $g$ to be non-zero and non-negative, this value is not $0$. It thus follows that $A \notin \BDO$.
\end{itemize}
\end{example}

\section{Discussion}
Our results immediately lead to some open questions, which we want to sketch here.

Since the Berezin transform $A \mapsto \widetilde{A}$ on $\mathcal L(F_{(k)}^2)$ is no longer injective when $k \geq 2$, one can of course wonder if it is still of some practical use besides Theorem \ref{thm:compact_Toeplitz_operators} (cf.\ \cite[Question 32]{hagger2023}). Indeed, it might very well be possible that it can still be used to characterize properties such as compactness, up to the obvious obstacles. That is, one could pose the following question:
\begin{question} \label{Question1}
Let $A \in \mathcal C_1(F_{(k)}^2)$ such that $\widetilde{A} \in C_0(\mathbb C)$. Can we then conclude
\[A \in \mathcal K(F_{(k)}^2) + \overline{\operatorname{span}} \{ W_\xi: ~L_{k-1}^0(|\xi|^2) =0\}?\]
\end{question}
Note that the closure is taken in weak$^\ast$ topology so that the space $\overline{\operatorname{span}} \{ W_\xi: ~L_{k-1}^0(|\xi|^2) =0\}$ is exactly the kernel of the Berezin transform (see Theorem \ref{thm:kernel_berezin_trafo}). From Theorem \ref{thm:compact_Toeplitz_operators} we know that the answer is positive for Toeplitz operators. Let us, at least for $k = 2$, add a short discussion which indicates that the answer to Question \ref{Question1} might be positive in general.

Let $A \in \mathcal C_1(F_{(2)}^2)$ with $\widetilde{A} \in C_0(\mathbb C)$. Recall that the only zeros of $\mathcal F_W(k_{0, (2)} \otimes k_{0, (2)})(\xi)$ are at $|\xi| = 1$. Let $\varepsilon > 0$ and fix an operator $B_\varepsilon \in \mathcal T^1(F_{(2)}^2)$ such that $\mathcal F_W(B_\varepsilon)(\xi) \neq 0$ for $|\xi| = 1$ and $\operatorname{supp}(\mathcal F_W(B_\varepsilon)) \subseteq \overline{B}(S^1, \varepsilon)$. Here, $\overline{B}(S^1, \varepsilon)$ consists of all $\xi \in \mathbb C$ such that $\operatorname{dist}(\xi, S^1) \leq \varepsilon$, where $S^1$ is the unit circle. Note that such an operator $B_\varepsilon$ always exists: Let $f \in \mathcal S(\mathbb C)$ such that $f(\xi) \neq 0$ for $|\xi| = 1$ and $\operatorname{supp}(f) \subseteq \overline{B}(S^1, \varepsilon)$, then $\mathcal F_W^{-1}(f)$ is in trace class and satisfies the required properties.

Let $\varepsilon > 0$ and $f \in L^1(\mathbb C)$ such that $\| A - f \ast A\|_{op} < \varepsilon$ (which is always possible by Lemma \ref{lem:approximation}). Now, note that 
$$ \{ \xi \in \mathbb C: \mathcal F_W(k_{0, (2)} \otimes k_{0, (2)})(\xi) = 0\} \cap \{ \xi \in \mathbb C: \mathcal F_W(B_\varepsilon)(\xi) = 0\} = \emptyset. $$
Therefore, by a simple modification of Wiener's approximation theorem for operators (cf.\ \cite[Proposition 3.5]{werner84} for the standard version), the closed $\alpha$-invariant subspace of $L^1(\mathbb C)$ generated by $(k_{0, (2)} \otimes k_{0, (2)}) \ast (k_{0, (2)} \otimes k_{0, (2)})$ and $B_\varepsilon \ast B_\varepsilon$ is all of $L^1(\mathbb C)$. Hence, we can find coefficients $c_j$ and $d_\ell$ such that
\begin{align*}
    \Bigg\|f - \sum_j c_j \alpha_{z_j}(k_{0, (2)} \otimes k_{0, (2)} \ast k_{0, (2)} \otimes k_{0, (2)}) - \sum_\ell d_\ell \alpha_{w_\ell}(B_\varepsilon \ast B_\varepsilon)\Bigg\|_1 < \varepsilon.
\end{align*}
With these coefficients, we now receive
\begin{align*}
    \Bigg\| A - \sum_j c_j \alpha_{z_j}(k_{0, (2)} \otimes k_{0, (2)}) \ast \widetilde{A} - \sum_\ell d_\ell \alpha_{w_\ell}(B_\varepsilon \ast B_\varepsilon) \ast A \Bigg\|_{op} < \varepsilon(1 + \| A\|_{op}).
\end{align*}

While the first sum is clearly contained in $\mathcal K(F_{(2)}^2)$ (see Theorem \ref{thm:compact_Toeplitz_operators}, for instance), the second sum generally is not. The second sum is contained in
\begin{align*}
    S_{\overline{B}(S^1, \varepsilon)} :=\overline{\operatorname{span}} \{ W_\xi: ~\xi \in \overline{B}(S^1, \varepsilon)\}  \cap \mathcal C_1(F_{(2)}^2),
\end{align*}
where the closure is taken in weak$^\ast$ topology. Note that membership of $\sum_\ell d_\ell \alpha_{w_\ell}(B_\varepsilon \ast B_\varepsilon) \ast A$ in $S_{\overline{B}(S^1, \varepsilon)}$ uses the fact that an annulus is a set of spectral synthesis. This follows from the fact that closed discs as well as complements of open discs are sets of spectral synthesis, cf.\ \cite[Theorem 2.7.10]{Reiter_Stegeman2000}. Therefore, an annulus is a set of spectral synthesis by \cite[Theorem 2]{Muraleedharan_Parthasarathy}. Since spectral synthesis is equivalent to quantum spectral synthesis (cf.\ \cite[Corollary 4.4]{werner84} or \cite[Theorem 2.2]{Fulsche_Rodriguez2023}) it follows that 
\begin{equation} \label{eq:qss}
\overline{\operatorname{span}} \{ W_\xi: ~\xi \in \overline{B}(S^1, \varepsilon)\} = \set{B \in \Lc(F^2_{(2)}) : \supp\Fc_W(B) \subseteq \overline{B}(S^1, \varepsilon)}
\end{equation}
and hence $B_\varepsilon \ast B_\varepsilon \ast A \in S_{\overline{B}(S^1, \varepsilon)}$ for every $\varepsilon > 0$. We obtain that $A \in \overline{\mathcal K(F_{(2)}^2) + S_{\overline{B}(S^1, \varepsilon)}}$ for every $\varepsilon > 0$, where the closure is taken in the norm topology. In particular, we obtain that $\widetilde{A} \in C_0(\mathbb C)$ implies
\begin{align*}
    A \in \bigcap_{\varepsilon > 0} \overline{\mathcal K(F_{(2)}^2) + S_{\overline{B}(S^1,\varepsilon)}}.
\end{align*}
for every $A \in \mathcal C_1(F_{(2)}^2)$. Moreover, Eq.\ \eqref{eq:qss} shows that
\[\bigcap_{\varepsilon > 0} S_{\overline{B}(S^1,\varepsilon)} = \overline{\operatorname{span}} \{ W_\xi: ~\xi \in S^1\} \cap \mathcal C_1(F_{(2)}^2) = \overline{\operatorname{span}} \{ W_\xi: ~L_1^0(|\xi|^2) =0\} \cap \mathcal C_1(F_{(2)}^2),\]
but it is not quite clear to us if this is already enough to answer Question \ref{Question1} affirmatively in case $k = 2$.

Besides the characterizations of $\mathcal C_1$ discussed in Section \ref{sec:4}, further characterizations are available for $k = 1$. It is known that $\mathcal C_1(F_{(1)}^2)$ equals $\overline{ \{ T_{f, (1)}: f \in \operatorname{BUC}(\mathbb C)\}}$ as well as $C^\ast(\{ T_{f, (1)}: f \in \operatorname{BUC}(\mathbb C)\})$. Let us again restrict our discussion to the case $k = 2$. By Corollary \ref{cor:Weyl_not_Toeplitz}, the Weyl operators $W_z$ with $|z| = 1$ are not contained in $\{ T_{f, (2)}: ~f \in L^\infty(\mathbb C)\}$. It is not clear if they are contained in the uniform closure of the set. By Proposition \ref{prop:toeplitz_quant_wstar_dense}, all Weyl operators are at least contained in the weak$^\ast$ closure.

Nevertheless, using that most Weyl operators can be written as Toeplitz operators, cf.\ Remark \ref{rem:Weyl_as_Toeplitz}, $\{ T_{f, (2)} T_{g, (2)}: ~f, g \in \operatorname{BUC}(\mathbb C)\}$ is easily seen to contain all Weyl operators. Therefore, there is at least no obvious obstacle preventing $C^\ast(\{ T_{f, (2)}: ~f\in \operatorname{BUC}\})$ from being all of $\mathcal C_1(F_{(2)}^2)$. But the problem remains open:
\begin{question}
Does $\overline{ \{ T_{f, (k)}: f \in \operatorname{BUC}(\mathbb C)\}}$ equal $\mathcal C_1(F_{(k)}^2)$ for $k \geq 2$? How about $C^\ast(\{ T_{f, (k)}: f \in \operatorname{BUC}(\mathbb C)\})$?
\end{question}

Our Question 1 above is essentially a somewhat polished version of \cite[Question 32]{hagger2023}. Question 2 is the part of \cite[Question 33]{hagger2023} that could not be answered in Section \ref{sec:4} above. We note once again that \cite[Question 31]{hagger2023} was answered in the negative in Remark \ref{rem:Question31}.

Finally, to conclude this discussion we briefly mention that the formalism of quantum harmonic analysis developed in Section \ref{sec:reducible} not only applies to operators on polyanalytic Fock spaces but also to operators acting on the vector-valued standard Fock space $F^2(\mathbb C^n, \mathbb C^k)$, that is, on the space of entire functions on $\mathbb C^n$ with values in $\mathbb C^k$ such that $\frac{1}{\pi^n}\int_{\mathbb C^n} \| f(z)\|^2 e^{-|z|^2}~dz < \infty$. This space can be written in a natural way as the orthogonal direct sum of $k$ copies of the scalar valued standard Fock space on $\mathbb C^n$: $F^2(\mathbb C^n, \mathbb C^k) = \oplus_{j=1}^k F^2(\mathbb C^n)$. It is now straightforward to apply the results of Section \ref{sec:reducible}. For example, if we want to study Toeplitz operators with matrix symbols on this space, we can do the following: For $F \in L^\infty(\mathbb C^n, \mathbb C^{k \times k})$ the Toeplitz operator $T_F$ applied to $g \in F^2(\mathbb C^n, \mathbb C^k)$ is equal to $T_F g := P(Fg)$, where $Fg$ is the usual matrix-vector product and $P$ acts entrywise as the orthogonal projection from $L^2(\mathbb C^n, \pi^{-n}e^{-|z|^2}\mathrm{d}z)$ onto $F^2(\mathbb C^n)$. It is straightforward to check that $T_F$ can be written as $A_0 \ast F$, where $A_0 = \frac{1}{\pi} (1 \otimes 1)_{\ell, j = 1, \dots, k}$. This implies that the correspondence results apply to these operators. In particular, $T_F$ is compact if and only if the Berezin transform of each entry of the symbol matrix is in $C_0(\mathbb C^n)$. Of course, there are more direct ways to obtain these results from the scalar-valued case, so we shall not elaborate on this in more detail.

\section*{Acknowledgement}

We would like to thank the referees for their careful reading and their valuable suggestions, which significantly improved the quality of this paper.

\bibliographystyle{abbrv}
\bibliography{main}

\begin{thebibliography}{10}

\bibitem{Abreu2010}
L.~D. Abreu.
\newblock {On the structure of Gabor and super Gabor spaces}.
\newblock {\em Monatsh. Math.}, 161:237--253, 2010.

\bibitem{Abreu2010b}
L.~D. Abreu.
\newblock Sampling and interpolation in {B}argmann-{F}ock spaces of
  polyanalytic functions.
\newblock {\em Appl. Comput. Harmon. Anal.}, 29(3):287--302, 2010.

\bibitem{Abreu_Balazs_deGosson_Mouayn2015}
L.~D. Abreu, P.~Balazs, M.~de~Gosson, and Z.~Mouayn.
\newblock {Discrete coherent states for higher Landau levels}.
\newblock {\em Ann. Physics}, 263:337–353, 2015.

\bibitem{Abreu_Groechenig_Romero2019}
L.~D. Abreu, K.~Gr\"{o}chenig, and J.~L. Romero.
\newblock {Harmonic analysis in phase space and finite Weyl-Heisenberg
  ensembles}.
\newblock {\em J. Stat. Phys.}, 174:1104–1136, 2019.

\bibitem{Arroyo_Sanchez_Hernandez_Lopez2021}
J.~L. Arroyo~Neri, A.~S\'{a}nchez-Nungaray, M.~Hern\'{a}ndez~Marroquin, and
  R.~R. L\'{o}pez-Mart\'{\i}nez.
\newblock Toeplitz operators with {L}agrangian invariant symbols acting on the
  poly-{F}ock space of {$\Bbb{C}^n$}.
\newblock {\em J. Funct. Spaces}, 2021.
\newblock 9919243.

\bibitem{Balk1991}
M.~B. Balk.
\newblock {\em Polyanalytic functions}, volume~63 of {\em Mathematical
  Research}.
\newblock Akademie-Verlag, Berlin, 1991.

\bibitem{Bauer_Fulsche2020}
W.~Bauer and R.~Fulsche.
\newblock {Berger-Coburn theorem, localized operators, and the Toeplitz
  algebra}.
\newblock In W.~Bauer, R.~Duduchava, S.~Grudsky, and M.~A. Kaashoek, editors,
  {\em Operator Algebras, Toeplitz Operators and Related Topics}, pages 53--77.
  Springer International Publishing, Cham, 2020.

\bibitem{Bauer_Isralowitz2012}
W.~Bauer and J.~Isralowitz.
\newblock {Compactness characterization of operators in the Toeplitz algebra of
  the Fock space $F_\alpha^p$}.
\newblock {\em J. Funct. Anal.}, 263:1323--1355, 2012.

\bibitem{Behrndt_Holzmann_Lotoreichik_Raikov2022}
J.~Behrndt, M.~Holzmann, V.~Lotoreichik, and G.~Raikov.
\newblock The fate of {L}andau levels under {$\delta$}-interactions.
\newblock {\em J. Spectr. Theory}, 12:1203–1234, 2022.

\bibitem{Berge_Berge_Luef_Skrettingland2022}
E.~Berge, S.~M. Berge, F.~Luef, and E.~Skrettingland.
\newblock Affine quantum harmonic analysis.
\newblock {\em J. Funct. Anal.}, 282:109327, 2022.

\bibitem{Constales_Faustino_Krausshar2011}
D.~Constales, N.~Faustino, and R.~Krau\ss{}har.
\newblock Fock spaces, {L}andau operators and the time-harmonic {M}axwell
  equations.
\newblock {\em J. Phys. A}, 44:135303, 31, 2011.

\bibitem{Folland1989}
G.~B. Folland.
\newblock {\em Harmonic Analysis in Phase Space. ({AM}-122)}.
\newblock Princeton University Press, 1989.

\bibitem{Fulsche2020}
R.~Fulsche.
\newblock {Correspondence theory on $p$-Fock spaces with applications to
  Toeplitz algebras}.
\newblock {\em J. Funct. Anal.}, 279:108661, 2020.

\bibitem{Fulsche_Galke2023}
R.~Fulsche and N.~Galke.
\newblock {Quantum Harmonic Analysis on locally compact abelian groups}.
\newblock {\em arXiv preprint 2308.02078}, 2023.

\bibitem{Fulsche_Rodriguez2023}
R.~Fulsche and M.~Rodriguez~Rodriguez.
\newblock {Commutative $G$-invariant Toeplitz $C^*$ algebras on the Fock space
  and their Gelfand theory through Quantum Harmonic Analysis}.
\newblock {\em to appear in J.~Operator Theory, arXiv preprint 2307.15632},
  2023.

\bibitem{Grochenig_Lyubarskii2009}
K.~Gr\"ochenig and Y.~Lyubarskii.
\newblock Gabor (super)frames with {H}ermite functions.
\newblock {\em Math. Ann.}, 345(2):267--286, 2009.

\bibitem{hagger2021}
R.~Hagger.
\newblock Essential commutants and characterizations of the {T}oeplitz algebra.
\newblock {\em J. Operator Theory}, 86:125--143, 2021.

\bibitem{hagger2023}
R.~Hagger.
\newblock Toeplitz and related operators on polyanalytic {F}ock spaces.
\newblock In E.~Basor, A.~B\"ottcher, T.~Ehrhardt, and C.~A. Tracy, editors,
  {\em Toeplitz Operators and Random Matrices: In Memory of Harold Widom},
  pages 401--425. Birkh\"auser, Cham, 2023.

\bibitem{HaggerSeifert}
R.~Hagger and C.~Seifert.
\newblock Limit operators techniques on general metric measure spaces of
  bounded geometry.
\newblock {\em J. Math. Anal. Appl.}, 489:124180, 2020.

\bibitem{Haimi_Hedenmalm2013}
A.~Haimi and H.~Hedenmalm.
\newblock {The polyanalytic Ginibre ensembles}.
\newblock {\em J. Stat. Phys.}, 153:10–47, 2013.

\bibitem{Halvdansson2023}
S.~Halvdansson.
\newblock Quantum harmonic analysis on locally compact groups.
\newblock {\em J. Funct. Anal.}, 285:110096, 2023.

\bibitem{Hewitt_Ross2}
E.~Hewitt and K.~A. Ross.
\newblock {\em {Abstract Harmonic Analysis II}}, volume 152 of {\em Grundlehren
  der mathematischen Wissenschaften}.
\newblock Springer Verlag, 1970.

\bibitem{IsralowitzMitkovskiWick}
J.~Isralowitz, M.~Mitkovski, and B.~Wick.
\newblock {Localization and compactness in Bergman and Fock spaces}.
\newblock {\em Indiana Univ. Math. J.}, 64:1553--1573, 2015.

\bibitem{Keller_Luef2021}
J.~Keller and F.~Luef.
\newblock {Polyanalytic Toeplitz operators: isomorphisms, symbolic calculus and
  approximation of Weyl operators}.
\newblock {\em J. Fourier Anal. Appl.}, 27:47, 2021.

\bibitem{keyl_kiukas_werner16}
M.~Keyl, J.~Kiukas, and R.~Werner.
\newblock Schwartz operators.
\newblock {\em Rev. Math. Phys.}, 28, 2016.

\bibitem{Kiukas_Lahti_Schultz_Werner2012}
J.~Kiukas, P.~Lahti, J.~Schultz, and R.~F. Werner.
\newblock {Characterization of informational completeness for covariant phase
  space observables}.
\newblock {\em J. Math. Phys.}, 53:102103, 2012.

\bibitem{Luef_Skrettingland2021}
F.~Luef and E.~Skrettingland.
\newblock {A Wiener tauberian theorem for operators and functions}.
\newblock {\em J. Funct. Anal.}, 280:108883, 2021.

\bibitem{Maximenko_Telleria2020}
E.~Maximenko and A.~M. Teller\'ia-Romero.
\newblock Radial operators on polyanalytic {Bargmann--Segal--Fock} spaces.
\newblock In W.~Bauer, R.~Duduchava, S.~Grudsky, and M.~A. Kaashoek, editors,
  {\em {Operator Algebras, Toeplitz Operators and Related Topics}}, volume 279
  of {\em Oper. Theory Adv. Appl.}, pages 277--305. Birkh\"auser, Basel, 2020.

\bibitem{Mouayn2011}
Z.~Mouayn.
\newblock Coherent state transforms attached to generalized {B}argmann spaces
  on the complex plane.
\newblock {\em Math. Nachr.}, 284(14-15):1948--1954, 2011.

\bibitem{Muraleedharan_Parthasarathy}
T.~K. Muraleedharan and K.~Parthasarathy.
\newblock On unions and intersections of sets of synthesis.
\newblock {\em Proc. Amer. Math. Soc.}, 123:1213–1216, 1995.

\bibitem{Reiter_Stegeman2000}
H.~Reiter and J.~D. Stegeman.
\newblock {\em {Classical Harmonic Analysis on Locally Compact Groups}}.
\newblock Oxford University Press, 2nd edition, 2000.

\bibitem{Rozenblum_Vasilevski2019}
G.~Rozenblum and N.~Vasilevski.
\newblock {Toeplitz operators in polyanalytic Bergman type spaces}.
\newblock In P.~Kuchment and E.~Semenov, editors, {\em {Operator Algebras,
  Toeplitz Operators and Related Topics}}, volume 733 of {\em Contemp. Math.},
  pages 273--386. Amer. Math. Soc., Providence, RI, 2019.

\bibitem{Schur1929}
I.~Schur.
\newblock {Einige S\"{a}tze \"{u}ber Primzahlen mit Anwendungen auf
  Irreduzibilit\"{a}tsfragen. I}.
\newblock {\em Sitzungsber. Akad. Berlin 1929}, pages 125--136, 1929.

\bibitem{Vasilevski2000}
N.~Vasilevski.
\newblock {Poly-Fock spaces}.
\newblock In {\em Differential operators and related topics, Vol. I (Odessa
  1997)}, Oper. Theory Adv. Appl., page 371–386. Birkh\"auser, Basel, 2000.

\bibitem{werner84}
R.~Werner.
\newblock Quantum harmonic analysis on phase space.
\newblock {\em J. Math. Phys.}, 25:1404–1411, 1984.

\bibitem{Xia2015}
J.~Xia.
\newblock {Localization and the Toeplitz algebra on the Bergman space}.
\newblock {\em J. Funct. Anal.}, 269:781–814, 2015.

\bibitem{XiaZheng}
J.~Xia and D.~Zheng.
\newblock {Localization and Berezin transform on the Fock space}.
\newblock {\em J. Funct. Anal.}, 264:97--117, 2013.

\bibitem{Zhu}
K.~Zhu.
\newblock {\em Analysis on {F}ock {S}paces}, volume 263 of {\em Graduate Texts
  in Mathematics}.
\newblock Springer US, New York, 2012.

\end{thebibliography}

\vspace{1cm}

\begin{multicols}{2}

\noindent
Robert Fulsche\\
\href{fulsche@math.uni-hannover.de}{\Letter ~fulsche@math.uni-hannover.de}
\\
\noindent
Institut f\"{u}r Analysis\\
Leibniz Universit\"at Hannover\\
Welfengarten 1\\
30167 Hannover\\
GERMANY\\

\noindent
Raffael Hagger\\
\href{hagger@math.uni-kiel.de}{\Letter ~hagger@math.uni-kiel.de}
\\
\noindent
Mathematisches Seminar\\
Christian-Albrechts-Universität zu Kiel\\
Heinrich-Hecht-Platz 6\\
24118 Kiel\\
GERMANY

\end{multicols}
\end{document}